\documentclass[12pt]{amsart}
\usepackage{amsfonts}
\usepackage{latexsym,amsmath,amsthm,amssymb,amsfonts,mathrsfs}
\usepackage{graphicx}
\usepackage[usenames, dvipsnames]{color}
\usepackage{tikz}
\usepackage{float}
\usepackage{ulem}
\usepackage{caption}

\numberwithin{equation}{section}

\newtheorem{thm}{Theorem}[section]
\newtheorem{cor}[thm]{Corollary}

\newtheorem{lem}[thm]{Lemma}
\newtheorem{prop}[thm]{Proposition}
\newtheorem{dfn}[thm]{Definition}
\theoremstyle{remark}
\newtheorem{rem}{Remark}[section]

\newcommand{\stkout}[1]{\ifmmode\text{\sout{\ensuremath{#1}}}\else\sout{#1}\fi}

\begin{document}
	
	\thanks{The authors were partially supported by DFF Sapere Aude 7027-00110B, by CPH-GEOTOP-DNRF151 and by CF21-0680 from respectively the Independent Research Fund Denmark, the Danish National Research Foundation and the Carlsberg Foundation.}
	
	
	
	
	\title[Embedded self-shrinkers with rotational symmetry]{Entropy bounds, Compactness and Finiteness Theorems for Embedded self-shrinkers with rotational symmetry}
	
	\author[J.M.S. Ma,  A. Muhammad, N.M. M\o ller]{John Man Shun Ma,  Ali Muhammad, Niels Martin M\o ller}

	\begin{abstract} In this work, we study the space of complete embedded rotationally symmetric self-shrinking hypersurfaces in $\mathbb{R}^{n+1}$. First, using comparison geometry in the context of metric geometry, we derive explicit upper bounds for the entropy of all such self-shrinkers. Second, as an application we prove a smooth compactness theorem on the space of all such shrinkers. We also prove that there are only finitely many such self-shrinkers with an extra reflection symmetry. 
	\end{abstract}
	
	\date{}
	
	\maketitle
	
	\section{Introduction}
	An $n$-dimensional smooth hypersurface $\Sigma^n\subseteq\mathbb R^{n+1}$ is a {\sl self-shrinker} if 
	\begin{align*}
		H_\Sigma(x) = \frac 12 \langle x, \vec n\rangle, \ \ \text{ for all }x\in \Sigma.
	\end{align*}
	Here $H_\Sigma$ is the mean curvature of $\Sigma$ with respect to the outward unit normal $\vec n$. Given a self-shrinker, one obtains by scaling a one parameter family of hypersurfaces $\Sigma_t = \sqrt{-t} \Sigma$, $t\in (-\infty, 0)$ which solves the mean curvature flow (MCF) equation,
	\begin{align} \label{MCF eqn}
		\left( \frac{\partial  \Sigma_t}{\partial t}\right)^\perp = \vec H_{\Sigma_t},
	\end{align}
	where $\vec H_\Sigma = -H_\Sigma\vec n$ denotes the mean curvature vector of $\Sigma$.
	
	Most importantly, self-shrinkers serve as singularity models for the MCF: under the Type I condition on the singularity, Huisken \cite{Huisken} showed that a rescaling of a MCF around a singularity converges locally smoothly subsequentially to a self-shrinker, and proved that closed shrinkers with positive mean curvature are round spheres. Later Ilmanen \cite{Ilmanen} proved the subsequential weak convergence of the tangent flow of any MCF to a self-shrinking solution. 
	
	For $n=1$, all compact immersed self-shrinkers in $\mathbb R^2$ were found in \cite{AL}, the circle being the only embedded example. For $n=2$, Brendle \cite{Brendle} proved the long-standing conjecture that the round sphere of radius $2$ is the only closed embedded genus zero self-shrinker in $\mathbb R^3$. For higher genus, embedded examples are constructed in \cite{Ang}, \cite{DN}, \cite{KKM}, \cite{Nguyen}, \cite{Ketover}, \cite{moller}, \cite{BNS}. In general, the space of embedded self-shrinkers is not well-understood, even in the case of e.g. topological 2-tori in $\mathbb{R}^3$.
	
	In this paper we direct our attention to the class of complete embedded self-shrinking hypersurfaces in $\mathbb R^{n+1}$ with a rotational symmetry, for $n\geq 2$. Using a shooting method in ODEs, Angenent constructed in \cite{Ang} the first nontrivial self-shrinkers in $\mathbb R^{n+1}$ besides the round sphere, the generalized cylinders and the plane. These self-shrinkers constructed in \cite{Ang} are rotationally symmetric, embedded, diffeomorphic to $\mathbb S^{n-1}\times \mathbb S^1$ and are commonly called {\sl Angenent doughnuts}. They were used in a parabolic maximum principle argument in \cite{Ang} to prove that, when $n\geq 2$, mean curvature flows may develop thin neck-pinch singularities.
	
	In \cite{KM}, Kleene and the third named author proved a partial classification of all complete embedded rotationally symmetric self-shrinkers in any dimension (see also \cite{Song}). Mramor proved in \cite{Alex} several compactness and finiteness results on the space of all such shrinkers. It is conjectured that, at least in dimension 2, the Angenent doughnut (which when $n=2$ is topologically a torus) is unique and gives the only embedded self-shrinking torus in $\mathbb R^3$. This conjecture is still open, even in the rotationally symmetric case \cite{DLN}. 
	
	The goal of this paper is to study various properties of the space of all complete embedded rotationally symmetric self-shrinkers in $\mathbb R^{n+1}$. 
	
	For any hypersurface $\Sigma$ in $\mathbb R^{n+1}$, let $\lambda (\Sigma)$ be the entropy of $\Sigma$ defined in \cite{CM1} (see also Section \ref{section basic definition} for the definition). 
	
	\begin{thm} \label{upper bound on length}
		For each $n\ge 2$, there is a positive number $E_n$ such that 
		\begin{align*}
			1 \leq \lambda (\Sigma) \le E_n.
		\end{align*}
		for any complete embedded rotationally symmetric self-shrinker $\Sigma^n\subseteq\mathbb R^{n+1}$. 
	\end{thm}
	
	The constants $E_n$ we obtain in Theorem \ref{upper bound on length} are explicit (see (\ref{E_n constants})). For example, when $n=2$ we have $E_2 \sim 2.24759$, while the $2$-dimensional Angenent torus constructed in \cite{Ang} has entropy around $1.85122$, as computed numerically in \cite{yakov},\cite{HN}. We remark that for $n=1$, there is no upper entropy bound for the family of Abresch-Langer immersed self-shrinking curves \cite{AL}. We also remark that if we exclude the stationary plane, the lower bounds can be improved to $\lambda (\mathbb S^n)$ (\cite{CIMW}, \cite{BW}); if we consider only self shrinking doughnuts, the lower bound can be improved to $\lambda (\mathbb S^1) = \sqrt{2\pi/e} \sim 1.52035$, since they all have non-trivial fundamental groups \cite{HW}.
	
	The proof of Theorem \ref{upper bound on length} will make essential use of the fact that Angenent's Riemannian metrics have Gaussian curvatures bounded below by strictly positive constants (see Section \ref{section basic definition} for the definition of the Angenent metrics).
	
	As an application, our next theorem gives a smooth compactness result for the space of embedded rotationally symmetric self-shrinkers.
	
	\begin{thm} \label{smooth compactness}
		For each $n\ge 2$, the space of complete embedded rotationally symmetric self-shrinkers $\Sigma^n\subseteq\mathbb R^{n+1}$ is compact in the $C^\infty_\mathrm{loc}$-topology.
	\end{thm}
	
	Under some extra assumptions on the bounds on entropy and genus, there are already several smooth compactness results for self-shrinkers. Colding and Minicozzi proved in \cite{CM2} the smooth compactness of the set of all complete embedded self-shrinkers $\Sigma^2\subseteq\mathbb R^3$ with bounded genus and Euclidean volume growth (see also \cite[Theorem 1.4]{DX}). Later Sun and Wang proved in \cite{SunWang} a similar compactness theorem for embedded self-shrinkers in $\mathbb R^3$ with fixed genus and uniformly bounded entropy. In particular, the $n=2$ case of Theorem \ref{smooth compactness} follows from the main theorem in \cite{SunWang} and Theorem \ref{upper bound on length}. We remark that there are more compactness results for two dimensional self-shrinkers in general, even in higher codimension \cite{ChenMa1}. In the rotationally symmetric situation, Mramor \cite{Alex} proved several compactness results on the space of compact embedded rotationally symmetric self-shrinkers in $\mathbb R^{n+1}$ with various assumptions on $n$, bounds on entropy and convexity of the profile curves. Theorem \ref{smooth compactness} is a natural generalization of the results therein.
	
	Theorem \ref{smooth compactness} has several consequences, which include an index upper bound (Corollary \ref{cor index upper bound}), finiteness of the set of possible entropy values (Corollary \ref{finiteness of entropy value}) and the following finiteness theorem. 
	
	\begin{thm} \label{finitely many symmetric one}
		For each $n\ge 2$, up to rigid motions there are only finitely many complete embedded rotationally symmetric self-shrinkers in $\mathbb R^{n+1}$ which are symmetric with respect to the hyperplane perpendicular to the axis of rotation. 
	\end{thm}
	
	We remark that embeddedness is necessary:  there are infinitely many immersed rotationally symmetric self-shrinkers constructed in \cite{DK} with this extra reflection symmetry. 
	
	In Section \ref{section basic definition}, we recall the basic definitions and results needed. In Section \ref{upper bound section} we prove Theorem \ref{upper bound on length}. In Section \ref{section smooth compactness} we prove Theorem \ref{smooth compactness} and \ref{finitely many symmetric one}. 
	
	\section{Background} \label{section basic definition}
	\subsection{Entropy and Self-shrinkers} We follow the notations in \cite{CM1}. Let $\Sigma \subset \mathbb R^{n+1}$ be an $n$-dimensional properly embedded hypersurface. For each $x_0\in \mathbb R^{n+1}$, $t_0>0$, define the $F$-functional 
	\begin{align*}
		F_{x_0, t_0} (\Sigma):= \frac{1}{(4\pi t_0)^{n/2}} \int_\Sigma e^{\frac{-|x-x_0|^2}{4t_0}} d\mu,
	\end{align*}
	where $d\mu$ is the volume form of $\Sigma$. The entropy of $\Sigma$ is defined as 
	\begin{align*}
		\lambda (\Sigma) = \sup_{x_0, t_0} F_{x_0, t_0} (\Sigma). 
	\end{align*}
	Using Huisken's monotonicity formula \cite{Huisken}, it was proved in \cite{CM1} that if $\{\Sigma_t\}_{t\in I}$ satisfies the MCF equation (\ref{MCF eqn}), then the entropy $t \mapsto \lambda (\Sigma_t)$ is non-increasing, and is constant if and only if $\Sigma_t$ is self-shrinking. We recall the following lemma proved in \cite[Section 7.2]{CM1}.
	
	\begin{prop} \label{equality between lambda and F_0,1}
		Let $\Sigma$ be a properly embedded self-shrinker. Then $\lambda(\Sigma) =F_{0,1} (\Sigma)$. 
	\end{prop}
	
	A hypersurface $\Sigma$ is a self-shrinker if and only if $\Sigma$ is critical with respect to the functional $\Sigma \mapsto F_{0,1} (\Sigma)$ \cite[Proposition 3.6]{CM1}. The second variation of $F_{0,1}$ at a self-shrinker is calculated in \cite[Section 4]{CM1}: for any normal variation $\Sigma_s$ of $\Sigma$ given by $f \vec n$, we have 
	\begin{align*}
		\frac{\partial^2 }{\partial s^2 } F_{0,1} (\Sigma_s) \bigg|_{s=0} = -\int_{\Sigma} f L f e^{-\frac{|x|^2}{4}} d\mu_\Sigma,
	\end{align*}
	where 
	\begin{align} \label{L operator definition}
		L = \Delta + |A|^2 -\frac 12 \langle x , \nabla (\cdot)\rangle + \frac 12
	\end{align}
	is the stability operator on $\Sigma$. It is also shown that all self-shrinkers in $\mathbb R^{n+1}$ are $L$-unstable (\cite{CM1}, see also \cite[Theorem 0.5]{CM2}). 
	
	\subsection{Rotationally symmetric self-shrinkers; Angenent doughnuts}
	Let $l$ be any line in $\mathbb{R}^{n+1}$ passing through the origin. A hypersurface $\Sigma$ of $\mathbb{R}^{n+1}$ is {\sl rotationally symmetric} with respect to $l$ if $R\Sigma = \Sigma$ for all rotations $R \in SO(n+1)$ fixing $l$.
	Assume $n\ge 2$ and let $\mathbb S^{n -1}\subset \mathbb R^n$ be the $(n-1)$-dimensional unit sphere. We denote the upper half-plane by
	\begin{align*}
		\mathbb H = \{ (x, r) \in \mathbb R^2 : r >0\}.
	\end{align*}
	
	\begin{dfn}
		Let $I$ be any interval in $\mathbb R$, or $I=\mathbb S^1$. Let $\sigma : I \to \mathbb H$ be a smooth embedding. Then the embedded hypersurface $\Sigma_\sigma$ in $\mathbb R^{n+1}$ with profile curve $\sigma(s) = (x(s), r(s))$ is given by 
		\begin{align*}
			\Sigma_\sigma = \{ (x(s), \omega r(s)) : \omega\in \mathbb S^{n-1}, s\in I\}.
		\end{align*}
	\end{dfn}
	
	The hypersurface $\Sigma_\sigma$ is rotationally symmetric with respect to $\ell = \mathbb{R} e_1$.
	
	\begin{prop}
		Given a profile curve $\sigma$. Then $\Sigma_\sigma$ is a self-shrinker if and only if $\sigma$ is a geodesic in $(\mathbb H, g_A)$, where $g_A$ is the incomplete {\bf Angenent metric} given by 
		\begin{align}\label{dfn of ang metric}
			g_A = r^{2(n-1)} e^{-\frac{x^2 + r^2}{2}} (dx^2 + dr^2).
		\end{align} 
	\end{prop}
	Direct calculations give
	\begin{align*}
		F_{0,1} (\Sigma_\sigma) = (4\pi)^{-n/2} \omega_{n-1} L_A(\sigma), 
	\end{align*}
	where $\omega_{n-1}$ is the surface area of $\mathbb S^{n-1}$ and $L_A(\sigma)$ is the length of $\sigma$ in $(\mathbb H, g_A)$. Hence if $\Sigma_\sigma$ is a self-shrinker, 
	\begin{align} \label{relation between entropy and length}
		\lambda (\Sigma_\sigma) = (4\pi)^{-n/2} \omega_{n-1} L_A(\sigma)
	\end{align}
	by Proposition \ref{equality between lambda and F_0,1}.

	Using a shooting method in ODEs, Angenent constructed in \cite{Ang} a compact embedded rotationally symmetric self-shrinker in $\mathbb R^{n+1}$ for each $n\ge 2$. The profile curve of the examples in \cite{Ang} are convex and symmetric with respect to the $r$-axis. The self-shrinkers so constructed are called {\sl Angenent doughnuts}. More recently, Drugan and Nguyen in \cite{DN} used a geometric flow to construct compact embedded rotationally symmetric self-shrinkers in $\mathbb R^{n+1}$ with the same property of the profile curve for all $n\ge 2$. It is not known if both constructions in \cite{Ang} and \cite{DN} resulted in the same self-shrinkers. On the other hand, Kleene and the third named author proved the following partial classification, which we will be making use of in the present paper:
	
	\begin{thm}{\cite[Theorem 2]{KM}} \label{partial classification of KM}
		Let $\Sigma$ be a complete embedded rotationally symmetric self-shrinker in $\mathbb R^{n+1}$. Then up to rigid motion, $\Sigma$ is either 
		\begin{itemize}
			\item [(i)] the hyperplane $\mathcal P = \{0\}\times \mathbb R^n$,
			\item [(ii)] the round sphere $\mathcal S$ of radius $\sqrt {2n}$, 
			\item [(iii)] the round cylinder $\mathcal C =  \mathbb R \times \mathbb S^{n-1} $ with radius $\sqrt{2(n-1)}$, or
			\item [(iv)] diffeomorphic to $ \mathbb S^1\times \mathbb S^{n-1} $.  
		\end{itemize}
	\end{thm}
	Note that in the last case of Theorem \ref{partial classification of KM}, $\Sigma = \Sigma_\sigma$ for some simple closed geodesic $\sigma$ in $(\mathbb H, g_A)$. This theorem is essential to the proof of Theorem \ref{smooth compactness}. 
	As a first simple but useful consequence, note that if a complete rotationally symmetric self-shrinker is embedded, it is automatically properly embedded.
	
	\section{Entropy bound for embedded rotationally symmetric self-shrinkers} \label{upper bound section}
	In this section we prove Theorem \ref{upper bound on length}. By a direct calculation using (\ref{dfn of ang metric}), the Gauss curvature of the Angenent metric $g_A$ is 
	\begin{align*}
		K = \frac{r^2+ (n-1)}{r^{2n}} e^{\frac{x^2 + r^2}{2}}.
	\end{align*}
	Note that $K$ is strictly positive. For each $n$, let $\kappa_n$ be the (positive) minimum of $K$. By simple calculus, one can find
	\begin{align*} 
		\kappa_n = \frac{y_n+(n-1)}{y^n_n} e^\frac{y_n}{2} , \ \ \text{ where } y_n = \frac{n-1 + \sqrt{9(n-1)^2 + 8(n-1)}}{2}.
	\end{align*}
	
	We will prove the following more general result.
	\begin{thm} \label{upper bound on length of minimal geodesic in a ball}
		Let $(M, g)$ be a $2$-dimensional Riemannian manifold so that $M$ is homeomorphic to $\mathbb R^2$ and $K_g \ge \kappa_g >0$, where $K_g$ is the Gauss curvature of $(M, g)$. Then every simple closed geodesic in $(M, g)$ has length $\le 2\pi/\sqrt{\kappa_g}$. 
	\end{thm}
	
	Theorem \ref{upper bound on length} follows directly from Theorem \ref{upper bound on length of minimal geodesic in a ball} and (\ref{relation between entropy and length}). Indeed, the constants $E_n$ in Theorem \ref{upper bound on length} are given by 
	\begin{align} \label{E_n constants}
		E_n = \frac{2\pi \omega_{n-1}}{(4\pi)^{n/2} \sqrt{\kappa_n}}. 
	\end{align}
	
	
	In the Appendix we show that $2< E_n \le E_2$ for all $n\ge 2$ and 
	$$\lim_{n\to \infty} E_n =  \sqrt{\frac{4\pi}{3}} \sim 2.04665.$$ 
	
	When $M$ is homeomorphic to $\mathbb S^2$, Theorem \ref{upper bound on length of minimal geodesic in a ball} is a classical theorem in comparison geometry, where a proof can be found in \cite[Theorem 3.4.10]{Klingenberg}; in our situation, $(M, g)$ is non-compact and incomplete, and we are not able to find an exact reference in this generality. As a result, we provide a proof of Theorem \ref{upper bound on length of minimal geodesic in a ball} using the globalization theorem in metric geometry \cite{BBI}. 
	
	Let $\sigma$ be a simple closed geodesic in $(M, g)$. Since $M$ is homeomorphic to $\mathbb R^2$, by the Jordan curve theorem $\sigma$ divides $M$ into two connected components, where exactly one of them has compact closure. 
	
	\begin{dfn} \label{dfn of sigma and omega}
		Given a simple closed geodesic $\sigma$ in $(M, g)$, let $\Omega$ be the compact domain in $(M, g)$ with $\partial \Omega = \operatorname{Im} (\sigma)$. 
	\end{dfn}
	
	\begin{lem} \label{shorest geo in Omega}
		Let $p, q\in \Omega$. Then there is a simple geodesic $\gamma$ in $\Omega$ joining $p, q$, which is shortest among all piecewise $C^1$ curves in $\Omega$ joining $p$ and $q$. Moreover, 
		\begin{itemize}
			\item [(i)] if one of $p, q$ is in the interior of $\Omega$, $\gamma$ also lies in the interior of $\Omega$ (except possibly at the other end point), 
			\item [(ii)] if both $p, q$ are in $\operatorname{Im}(\sigma) = \partial \Omega$, then either $\gamma$ lies completely in $\partial \Omega$, or the interior of $\gamma$ lies inside the interior of $\Omega$.
		\end{itemize}
	\end{lem}
	
	\begin{proof}
		Let $p, q\in \Omega$. Since the case $p=q$ is trivial, we assume $p\neq q$. Let 
		\begin{align*}
			d^\Omega (p, q) = \inf_\gamma L(\gamma),
		\end{align*}
		where $L(\gamma) = \int \sqrt{g(\dot\gamma, \dot\gamma)}$ is the length of $\gamma$ and the infimum is taken among all piecewise $C^1$ curves $\gamma : [0,1]\to \Omega$ so that $\gamma(0) =p$ and $\gamma(1) =q$. It is easy to check that $d^\Omega$ is a metric on $\Omega$. In particular, $d^\Omega (p, q) >0$. 
		
		Let $\gamma_j : [0,1]\to \Omega$ be a sequence of piecewise $C^1$ curves, parametrized proportional to arc length, joining $p$, $q$ so that $L(\gamma_j) \to d^\Omega(p, q)$ as $j\to \infty$. Since $\Omega$ is compact, by passing to a subsequence if necessary, we may assume that $(\gamma_j)$ converges uniformly in $d^\Omega$ to a Lipschitz continuous curve $\gamma: [0,1]\to \Omega$ as $j\to\infty$. If $\gamma(t_0) = q$ for some $t_0< 1$, we replace $\gamma$ by $\gamma |_{[0,t_0]}$. Hence we can assume $\gamma ^{-1} (\{q\}) = \{1\}$. 
		
		First we prove (i). Assume that $p \in \Omega \setminus \partial \Omega$. We claim that 
		\begin{align} \label{gamma stay inside Omega proof}\gamma(t) \notin \partial \Omega, \ \ \text{ for all }t<1. 
		\end{align}
		We argue by contradiction: if not, let $t_0$ be the infimum of the set $\gamma^{-1} (\partial \Omega\setminus\{q\})$. Since $p\notin \partial \Omega$, we have $t_0 >0$ and $\gamma([0,t_0))$ lies in the interior of $\Omega$. This together with the definition of $\gamma$ implies that $\gamma|_{[0,t_0]}$ is a geodesic. Let $U$ be a small geodesically convex neighborhood in $(M, g)$ centered at $\gamma (t_0)$ not containing $q$. Then there is $\epsilon>0$ such that $\gamma(t)$ lies in $U$ for all $t\in [t_0-\epsilon, t_0+\epsilon]$ and $\gamma(t_0\pm \epsilon) \neq \gamma(t_0)$. Let $\sigma_0$ be the shortest geodesic in $U$ connecting $\gamma (t_0 \pm \epsilon)$. Since $\partial \Omega \cap U$ is a portion of the geodesic $\sigma$ and $U$ is geodesically convex, $\sigma_0$ does not intersect with $\partial \Omega \cap U$ at more than one points. Hence the image of $\sigma_0$ except possibly at $\gamma(t_0+\epsilon)$ must lie in the connected component of $U\setminus \partial \Omega$ containing $\gamma(t_0 - \epsilon)$. As a result, the image of $\sigma_0$ also lies in $\Omega \cap U$. Since $\gamma$ is length minimizing, up to reparametrization we have $\gamma|_{[t_0-\epsilon, t_0 + \epsilon]} = \sigma_0$ and thus $\gamma|_{[t_0-\epsilon, t_0+ \epsilon]}$ is a smooth geodesic. Since $\gamma$, $\sigma$ are tangential at $\gamma(t_0)$ (recall that $\partial \Omega = \mathrm{Im}\sigma$) and both are geodesics, we have $\gamma = \sigma$ locally around $\gamma(t_0)$. This contradicts the choice of $t_0$ and thus (\ref{gamma stay inside Omega proof}) is shown. This immediately implies (i). 
		
		Next we prove (ii). Assume that $p, q\in \partial \Omega$. Let $(p_k)$ be a sequence of points in the interior of $\Omega$ converging to $p$. For each $k\in \mathbb N$, let $\gamma_k : [0,1]\to \Omega$ be a shortest geodesic in $\Omega$ joining $p_k$ to $q$ constructed in (i). By the smooth dependence of solutions to the geodesic equation and picking a subsequence if necessary, $(\gamma_k)$ converges smoothly to a geodesic $\gamma : [0,1]\to \Omega$ joining $p, q$. Using the triangle inequality
		\begin{align*}
			L(\gamma_k) = d^\Omega(p_k, q) \le d^\Omega (p_k, p) + d^\Omega(p, q)
		\end{align*}
		and taking $k\to \infty$, we have $L(\gamma) \le d^\Omega(p, q)$. Thus $\gamma$ is a length minimizing geodesic in $\Omega$. Then either $\gamma$ lies completely inside the interior of $\Omega$ away from the endpoints, or $\gamma$ touches $\partial \Omega$ at some point in $\partial\Omega\setminus \{p,q\}$, which implies that $\gamma$ lies completely in $\partial \Omega$ since both of them are geodesics. This finishes the proof of (ii). 
	\end{proof}
	
	We will need some definitions and notations from metric geometry. We use the reference \cite{BBI}. For the convenience of the reader, we summarize the basic facts that we shall need to prove Theorem \ref{upper bound on length of minimal geodesic in a ball}.
	
	A {\sl length space} $(X,d)$ is a metric space such that the metric $d$ can be obtained as a distance function associated to a length structure (see \cite{BBI} for a definition). The metric $d$ is called an {\sl intrinsic metric} in this case. If every pair of points $p,q$ in $X$ can be joined by a (possibly non-unique) shortest path, then the metric $d$ is called {\sl strictly intrinsic}. We recall that a shortest path is a curve $\gamma$ where the length $L(\gamma)$ is given by the distance between the endpoints of $\gamma$. A length space whose metric is strictly intrinsic is called a {\sl complete length space}.
	
	
	
	From Lemma \ref{shorest geo in Omega} we conclude 
	
	\begin{prop} \label{omega, d^omega is complete intrinsic metric space}
		$(\Omega, d^\Omega)$ is a complete length space. 
	\end{prop}
	
	A {\sl triangle} $\Delta pqr$ in $(X,d)$ is a set of points $\{p,q,r\}$ together with three shortest paths $[pq], [qr], [rp]$. The length of a triangle is the sum of the lengths of its sides. For each $\kappa >0$ and for each triangle $\Delta pqr$ in $X$ with length $< 2\pi / \sqrt\kappa$, we can associate a unique (up to an isometry) comparison geodesic triangle $\Delta \bar{p}\bar{q}\bar{r}$ in $\mathbb S^2_{1/\sqrt\kappa}$ with vertices $\{\bar{p},\bar{q},\bar{r}\}$ such that the corresponding sides of the geodesic triangle $\Delta \bar{p}\bar{q}\bar{r}$ have the same lengths as the sides of the triangle $\Delta pqr$. Here $\mathbb S^2_{1/\sqrt\kappa}$ is the $2$-sphere with radius $1/\sqrt{\kappa}$. For a triangle $\Delta pqr $ in $X$, we denote the angles by $\angle p, \angle q, \angle r$, and if confusion arises we will e.g write $\angle pqr$ for the angle $\angle q$. We shall not need the definition of an angle between two shortest paths in a length space, but one can show that on a Riemannian manifold $M$, if $c_1$ and $c_2$ are two geodesics starting at $p = c_1(0) = c_2(0)$, then the angle $\angle p \in [0,\pi]$ between the shortest paths $c_1$ and $c_2$ is equal to the usual Riemannian angle between $c_1$ and $c_2$. See Corollary 1A.7 in \cite{bridson} for a proof. For the comparison triangle  $\Delta \bar{p}\bar{q}\bar{r}$ in $\mathbb S^2_{1/\sqrt\kappa}$ we denote the angles by $\angle \bar{p}, \angle \bar{q}, \angle \bar{r}$. 
	
	There are several definitions of a \textit{space of curvature} $\geq \kappa$. We shall use the following \textit{angle comparison} definition. 
	\begin{dfn}\label{def:space_of_curvature_k}
		Let $X$ be a complete length space, and let $\kappa > 0$. We say that $X$ is a \textit{space of curvature $\geq \kappa$} if for any point $x \in X$ there is a neighborhood $U_x$ of $x$ such that for all triangles $\Delta pqr \subset U_x$ the corresponding angles satisfy the inequalities
		\begin{align}\label{eq:angle_comparison}
			\angle p \geq \angle \bar{p}, \quad \angle q \geq \angle \bar{q}, \quad \angle r \geq \angle \bar{r},
		\end{align}
		for a comparison triangle $\Delta \bar{p}\bar{q}\bar{r}$ in $\mathbb S^2_{1/\sqrt\kappa}$. Furthermore, for any two shortest paths $[ab]$ and $[cd]$ in $X$ where $c$ is an interior point on $[ab]$, the following holds $\angle acd + \angle bcd = \pi$.
	\end{dfn}
	\begin{rem}
		The last statement in Definition \ref{def:space_of_curvature_k} above can be summarized as “the sum of adjacent angles equals $\pi$". It is needed to prove the equivalence to other definitions of spaces of curvature $\geq \kappa$. See \cite[Section 4.3]{BBI}.
	\end{rem}
	
	\begin{prop} \label{Omega domega has curvature ge kappa}
		The complete length space $(\Omega, d^\Omega)$ has curvature $\ge \kappa_g$, where $\kappa_g$ is the lower bound of the Gaussian curvature of $(M, g)$ in Theorem \ref{upper bound on length of minimal geodesic in a ball}. 
	\end{prop}
	
	\begin{proof}
		Let $x\in \Omega$ and let $V_x$ be a geodesically convex neighborhood in $(M, g)$ centered at $x$. When $x$ is in the interior of $\Omega$, we choose $V_x\subset \Omega$. Let $U_x = \Omega \cap V_x$. For any $p, q\in U_x$, let $\gamma$ be the unique shortest geodesic in $V_x$ joining $p$ and $q$. Note that $\gamma$ must lie inside $U_x$: this holds when $x$ is in the interior of $\Omega$, since $U_x = V_x$. When $x\in \partial \Omega$, if $\gamma$ is not in $U_x$, there are $t_1<t_2$ such that $\gamma |_{(t_1, t_2)}$ lies in $V_x \setminus U_x$ and $\gamma(t_1), \gamma(t_2)$ are in $\partial \Omega \cap V_x$. Since $\partial \Omega \cap V_x$ is also a geodesic passing through $\gamma(t_1), \gamma(t_2)$, this contradicts the fact that $V_x$ is a geodesically convex neighborhood. 
		
		In particular, triangles of $(\Omega, d^\Omega)$ in  $U_x$ are also triangles of $(M, g) $ in $V_x$. Together with the assumption that $(M, g)$ has Gaussian curvature $K\ge \kappa _g >0$, we can argue as in the case for Riemannian manifolds without boundary to show that $(\Omega, d^\Omega)$ has curvature $\geq \kappa_g$; see Theorem 6.5.6. in \cite{BBI}.
	\end{proof}
	
	We will denote the inequalities in \eqref{eq:angle_comparison} by \textit{the angle comparison condition}. If $X$ is a complete length space, then Toponogov's globalization theorem \cite[Theorem 10.3.1]{BBI} globalizes this local curvature condition to the entirety of $X$, not only in a neighborhood $U_x$ of each point $x$: 
	\begin{thm}[Globalization Theorem]\label{thm:globalization}
		Let $X$ be a complete length space of curvature $\geq \kappa$ for some $\kappa>0$. Then the angle comparison condition \eqref{eq:angle_comparison} is satisfied for any triangle $\Delta pqr$ in $X$ for which there is a unique (up to isometry) comparison triangle $\Delta \bar{p}\bar{q}\bar{r}$ in $\mathbb S^2_{1/\sqrt\kappa}$.
		
	\end{thm}
	
	Let us remind that when $\kappa >0$, a comparison triangle only exists if the length of $\Delta pqr$ does not exceed $2\pi/\sqrt\kappa$. As mentioned earlier, we can associate a unique comparison triangle when the length of  $\Delta pqr$ is strictly less than $2\pi/\sqrt{\kappa}$. If the length is equal to $2\pi/\sqrt\kappa$, then we have two situations:
	\begin{itemize}
		\item All the sides have length strictly less than $\pi / \sqrt{\kappa}$. Then a comparison triangle is a unique great circle, i.e. all of its angles are equal to $\pi$.
		\item One of the sides has length equal to  $\pi / \sqrt{\kappa}$, say $[pq]$. The sum of lengths of the two other sides $[qr]$ and $[rp]$ is then equal to  $\pi / \sqrt{\kappa}$. In this case there does not exist a unique comparison triangle, but we can fix the comparison triangle $\Delta \bar{p}\bar{q}\bar{r}$ where the side $[\bar{p}\bar{q}]$ passes through the point $\bar{r}$.
	\end{itemize}
	Thus, with this convention, the conclusion of Theorem \ref{thm:globalization} holds for every triangle  $\Delta pqr$ in $X$ with length $\leq 2\pi/\sqrt{\kappa}$. In fact, as a corollary of theorem \ref{thm:globalization} one can show that any triangle in $X$ has length no greater than $2\pi/\sqrt{\kappa}$, where $X$ is a space of curvature $\geq \kappa$ for some $\kappa >0$. See \cite[Corollary 10.4.2.]{BBI}. We shall use this fact in the proof of Theorem \ref{upper bound on length of minimal geodesic in a ball} below.
	
	Before proving Theorem \ref{upper bound on length of minimal geodesic in a ball}, we recall the following elementary lemma in spherical geometry. We recall that a {\sl convex polygon} in the sphere is a polygon so that the interior angle at each vertex is less than or equal to $\pi$.
	\begin{lem}\label{lem:spherical_polygon}
		Let $P$ be a convex $n$-gon in $\mathbb S^2_{1/\sqrt{\kappa}}$, and denote the length of $P$ by $|P|$. Then $|P|\leq 2\pi/\sqrt{\kappa}$.
	\end{lem}
	
	
	
	\begin{proof} [Proof of Theorem \ref{upper bound on length of minimal geodesic in a ball}] The idea of this proof is similar to the proof presented in \cite[Theorem 3.4.10]{Klingenberg}. Let $\Omega$ be the compact domain bounded by the geodesic $\sigma: [0,1] \to M$ as defined in Definition \ref{dfn of sigma and omega}, so that $\sigma(0)=\sigma(1)$. By Proposition \ref{omega, d^omega is complete intrinsic metric space} and Proposition \ref{Omega domega has curvature ge kappa}, $(\Omega, d^\Omega)$ is a complete length space with curvature $\geq \kappa_g$. 
		
		Let $L_0 := L(\sigma)$ be the length of $\sigma$. Take numbers $t_1,t_2,t_3,t_4 \in I$, where $t_1=0, t_4=1$, such that each of the subarcs $\sigma|_{[t_1,t_2]}, \sigma|_{[t_2,t_3]}$ and $\sigma|_{[t_3,t_4]}$  has length $L_0/3$. Let $\Delta \sigma(t_1)\sigma(t_2)\sigma(t_3)$ be a triangle in $(\Omega, d^\Omega)$. By Proposition \ref{omega, d^omega is complete intrinsic metric space} each of the sides of this triangle either lies completely in $\partial \Omega$ or has interior contained in the interior of $\Omega$. Using \cite[Corollary 10.4.2.]{BBI} we deduce that the length of this triangle is not greater than $2\pi/\sqrt{\kappa_g}$. If the sides of the triangle $\Delta \sigma(t_1)\sigma(t_2)\sigma(t_3)$ coincide with the images of the subarcs $\sigma|_{[t_1,t_2]}, \sigma|_{[t_2,t_3]}$ and $\sigma|_{[t_3,t_4]}$, then we are done. 
		
		If not, then build new triangles 
		\begin{itemize}
			\item $\Delta \sigma(t_1)\sigma(t_{1,2})\sigma(t_2)$, 
			\item $\Delta \sigma(t_2)\sigma(t_{2,3})\sigma(t_3)$ and 
			\item $\Delta \sigma(t_3)\sigma(t_{3,1})\sigma(t_1)$,
		\end{itemize}
		where $t_{i,k} \in (t_i,t_k)$ are chosen such that $\sigma(t_{i,k})$ define midpoints of the corresponding subarcs. To each one of these triangles, we associate a unique comparison triangle in $\mathbb S^2_{1/\sqrt\kappa_g}$. We put the triangles together along their common sides and obtain a comparison $6$-gon $\mathcal{O}_6$ in $\mathbb S^2_{1/\sqrt\kappa_g}$. From Theorem \ref{thm:globalization} we know that the angles in the comparison triangles are not greater than the corresponding angles in $\Omega$. This implies that the angles of the vertices in $\mathcal{O}_6$ are not bigger than $\pi$. Hence $\mathcal{O}_6$ is convex. By Lemma \ref{lem:spherical_polygon}, the length of $\mathcal{O}_6$ satisfies $|\mathcal{O}_6|\leq 2\pi/\sqrt{\kappa_g}$. If the sides of the constructed 6-gon in $\Omega$ coincide with the arcs of $\sigma$, then we are done. 
		
		If not, then we continue to construct more triangles as above, and build the corresponding comparison $n$-gons $\mathcal{O}_n$ for increasingly large $n\in \mathbb{N}$. By Lemma \ref{lem:spherical_polygon} again, $|\mathcal{O}_n|\leq 2\pi/\sqrt{\kappa_g}$. For $n$ large enough, the arcs of $\sigma$ will be the unique shortest paths between the vertices of the constructed $n$-gon on $\Omega$. This follows from the fact that shortest paths in  $(\Omega, d^\Omega)$ are geodesics in $(M,g)$, and that each side of the $n$-gon will be contained in a geodesically convex neighborhood in $M$. In such a neighborhood, every two points are connected by a unique shortest path. This implies the desired bound since for $n$ large enough we have $L(\sigma)= |\mathcal O_n| \leq 2\pi /\sqrt{\kappa_g}$.

	\end{proof}

	\section{Compactness and Finiteness of Embedded\\ Self-Shrinkers With Rotational Symmetry} \label{section smooth compactness}
	In this section we prove Theorem \ref{smooth compactness}, \ref{finitely many symmetric one} and some other related results. We start with the following well known lemma.
	\begin{lem} \label{all self-shrinkers intersect}
		Let $\Sigma_1$, $\Sigma_2$ be two properly embedded self-shrinkers in $\mathbb R^{n+1}$ such that one of them is compact. Then $\Sigma_1$ and $\Sigma_2$ must intersect.
	\end{lem}
	
	When both self-shrinkers are compact, this is proved in \cite[Theorem 7.4]{WW}. Recently, it was proved in \cite{IPR} that any two properly embedded self-shrinkers that are sufficiently separated at infinity must intersect. See also \cite[Corollary C.4]{ CCMS} for the statement for $F$-stationary varifolds. 
	
	
	Next we restrict attention to complete embedded rotationally symmetric self-shrinkers in $\mathbb R^{n+1}$. By Theorem \ref{partial classification of KM}, there are only two types of non-compact examples - cylinders and hyperplanes through the origin, and all such pairs also intersect. Thus we have 
	
	\begin{lem} \label{all tori intersect}
		Any two complete embedded rotationally symmetric self-shrinkers in $\mathbb R^{n+1}$ intersect.
	\end{lem}
	
	Now we prove Theorem \ref{smooth compactness}.
	
	\begin{proof} [Proof of Theorem \ref{smooth compactness}] Let $(\Sigma_k)$ be a sequence of complete embedded rotationally symmetric self-shrinkers in $\mathbb R^{n+1}$ each with axis of rotation $\ell_k$. After taking a subsequence, a limit axis exists, to which each $\Sigma_k$ can be rotated, thus it is enough to consider the case where all axes of rotation are identical. By a further rotation, we assume the limit axis is $\ell =\mathbb{R} e_1\subseteq\mathbb{R}^{n+1}$.
		
		By Theorem \ref{partial classification of KM}, it suffices to assume that $\Sigma_k$ is a self-shrinking doughnut for each $k$. Thus $\Sigma_ k =\Sigma_{\sigma_k}$, where 
		$$\sigma _k: [-d_k, d_k] \to \mathbb H$$ 
		is a unit speed geodesic in $(\mathbb H, g_A)$ so that $L(\sigma_k)= 2d_k$ and $\sigma_k(-d_k) =\sigma_k(d_k)$. Since all self-shrinkers have entropy larger than or equal to $\lambda(\mathbb R^n) =  1$, we have $d_k \ge d>0$ for all $k$ by (\ref{relation between entropy and length}) for some dimensional constant $d$.
		
		Let $\sigma_a$ be the profile curve of the Angenent doughnut constructed in \cite{Ang}. By Lemma \ref{all tori intersect}, each $\sigma_k$ intersects $\sigma_a$. Reparametrizing each $\sigma_k$ if necessary, we assume that $\sigma_k (0) \in \operatorname{Im} \sigma_a$ for each $k$. Taking a subsequence if necessary, since $\operatorname{Im}\sigma_a$ is compact, we have
		\begin{align} \label{ sigma_n converges at a point}
			\sigma_k (0) \to p, \ \ \ \sigma_k'(0)\to v
		\end{align}
		as $k\to \infty$. Note that $\|v\|=1$ since each $\sigma_k$ is of unit speed. Let $\sigma_\infty : I \to \mathbb H$ be the maximally defined geodesic in $(\mathbb H, g_A)$ with $\sigma_\infty(0) = p$, $\sigma_\infty'(0) = v$, where $I$ is an open interval.
		
		For any $R>1$, let 
		$$K_R =[-R, R] \times [R^{-1}, R]  \subset \mathbb H.$$ 
		Then there is $R_0 >1$ so that $\operatorname{Im}(\sigma_a) \subset K_R$ for all $R \ge R_0$. For each $k\in \mathbb N$ and $R\ge R_0$, let $I_{k, R}$ be the connected component of $\sigma_k^{-1} (K_R)$ in $\mathbb [-d_k, d_k]$ containing $0$. Since each $\sigma_k$ is a geodesic in $(\mathbb H, g_A)$, by (\ref{ sigma_n converges at a point}) and the smooth dependence on initial data of the ODE within each $K_R$, $(\sigma_k|_{I_{k,R}})$ converges smoothly to $\sigma_{\infty}|_{I_R}$ in $K_R$, where $I_R$ is the connected component of $\sigma_\infty^{-1}(K_R)$ containing $0$.
		
		Since each $\sigma_k|_{I_{k, R}}$ is embedded and $(\sigma_k|_{I_{k, R}})$ converges smoothly to $\sigma_\infty |_{I_R}$, it is clear that $\sigma_\infty$ does not admit transverse self-intersection. Thus if $\sigma_\infty  (s) = \sigma_\infty (t)$ for some $s\neq t$, then $\sigma_\infty'(s) = \pm \sigma_{\infty}' (t)$ and this implies that $\sigma_\infty$ is periodic. In this case, $\sigma_\infty$ is a simple closed geodesic in $(\mathbb H, g_A)$ and $(\sigma_k)$ converges smoothly to $\sigma_\infty$. Thus we are done. 
		
		From now on we may therefore assume that $\sigma_\infty$ is not a closed geodesic. Since $(\sigma_k|_{I_{k, R}})$ converges smoothly to $\sigma_\infty$ in each $K_R$ and $\cup_{R\ge R_0} K_R = \mathbb H$, $\sigma_\infty$ is properly immersed. Since $\sigma_\infty$ is injective, $\Sigma_{\sigma_\infty}$ is a complete embedded rotationally symmetric self-shrinker in $\mathbb R^{n+1}$. By Theorem \ref{partial classification of KM}, $\Sigma_{\sigma_\infty}$ is either the plane $\mathcal P$, the sphere $\mathcal S$ or the cylinder $\mathcal C$. 
		
		First we assume that $\Sigma_{\sigma_\infty}$ is the sphere $\mathcal S$ and we will derive a contradiction. After that, we consider the hyperplane $\mathcal P$ and the cylinder $\mathcal C$ and point  out the necessary changes for the contradiction argument. 
		
		We split the argument into several lemmas.
		
		\begin{lem} \label{claim 1} There is $R_1 \ge R_0$ so that the following holds: for all $R \ge R_1$, there is $k_1 = k_1(R)$ so that for all $k\ge k_1$, $\operatorname{Im}\sigma_k\cap K_R$ contains a connected component different from $\sigma_k (I_{k,R})$ which intersects $K_{R_1}$. 
		\end{lem}
		
		\begin{proof} [Proof of Lemma \ref{claim 1}] First let $s\in (0,1)$ be small so that the scaling of the Angenent doughnut $s\Sigma_{\sigma_a}$ lies completely inside the sphere $\Sigma_{\sigma_\infty}=\mathcal S$: that is, $x^2 + r^2 < 2n$ for all $(x, r) \in s\operatorname{Im}\sigma_a$. Let $\delta = d_0 (s\Sigma_{\sigma_a}, \Sigma_{\sigma_\infty} \cup \ell)$, where $\ell \subset \mathbb{R}^{n+1}$ is the axis of rotation and $d_0$ is the Euclidean distance. Note that $\delta >0$. Let $R_1 = \max\{ R_0, 2\delta^{-1}, 3\sqrt{2n}\}$. For any $R >R_1$, since $\sigma_k|_{I_{k, R}}$ converges to $\sigma_\infty|_{I_{R}}$ in $K_R$ uniformly, there is $k_1(R)>0$ so that 
			\begin{align} \label{d_0(sigma_k(I_k, R), sigma_infty(I_R)) <delta/2}
				d_0(\sigma_k(I_{k, R}), \sigma_\infty(I_R)) <\frac{\delta}{2}
			\end{align} 
			for all $k\ge k_1$. By the choice of $\delta$, we have
			$$\sigma_k(I_{k, R}) \cap s(\operatorname{Im}\sigma_a) = \emptyset,\ \ \text{ for all } k\ge k_1.$$
			
			Now we fix $k\ge k_1$ and show that $\operatorname{Im}\sigma_k\cap K_R$ has more than one component which intersects $K_{R_1}$. Assume the contrary, then $ \operatorname{Im}\sigma_k\cap K_{R_1} = \sigma_k (I_{k, R})\cap K_{R_1}$. Let 
			$$K_{R_1}^\delta = K_{R_1} \setminus \{ (x, r)\in \mathbb 
			H\ |\ \sqrt{x^2+r^2} < \sqrt{2n}+\delta\}. $$
			Note that $K_{R_1}^\delta$ is connected. By (\ref{d_0(sigma_k(I_k, R), sigma_infty(I_R)) <delta/2}), $K^\delta_{R_1}$ is disjoint from $\operatorname{Im}\sigma_k$. Since $\sigma_k$ is a simple closed curve, by the Jordan curve theorem, its image divides $\mathbb H$ into two connected components, where exactly one of them is compact. There are two cases: 
			\begin{itemize}
				\item [(i)] $K_{R_1}^\delta$ lies in the compact component (see Figure 1). Let 
				$$\tilde\sigma_ a:= s\sigma_a + (2\sqrt{2n},0)$$ 
				be the horizontal translation of $s\sigma_a$ by $(2\sqrt{2n}, 0)$. Since $R_1 > 3\sqrt{2n}$ and $R^{-1}_1 < \delta$, $\tilde\sigma_ a$ lies completely inside $K_{R_1}^\delta$ and thus in the compact component (see Figure 1). The translation is horizontal, hence the MCF $\{ \Sigma^t_{\tilde\sigma_ a}\}$ starting at $\Sigma_{\tilde \sigma_a}$ is also self-shrinking, centered at $(2\sqrt{2n},0)$. Since $s<1$, it becomes extinct before the MCF $\{\Sigma^t_{\sigma_k}\}$ starting at $\Sigma_{\sigma_k}$ does. This implies that $\Sigma^{t}_{\tilde \sigma_a}$ intersects $\Sigma^t_{\sigma_k}$ for some $t>0$, which contradicts the maximum principle since $\Sigma_{\tilde \sigma_a}$ and $\Sigma_{\sigma_k}$ are disjoint.

				\item [(ii)] $K_{R_1}^\delta$ lies in the non-compact component (see Figure 2): let 
				$$H_{R_1}^\delta = \{ (x, r) :K_{R_1}\ |\  \sqrt{x^2+r^2}\le \sqrt{2n}-\delta\}.$$
				Since $ H_{R_1}^\delta$ is a subset of $K_{R_1}$ and is disjoint from $\sigma_k (I_{k, R})$, $H^\delta_{R_1}$ also lies in a component of $\mathbb H\setminus \operatorname{Im}(\sigma_k)$. Since $K^\delta_{R_1}$, $H^\delta_{R_1}$ are separated by $\sigma_k (I_{k, R})$, then $H^\delta_{R_1}$ and $K_{R_1}^\delta$ lie in different components of $\mathbb H \setminus  \operatorname{Im}\sigma_k$. Thus $H^\delta_{R_1}$ is in the compact component bounded by $\operatorname{Im}\sigma_k$. By the choice of $R_1$ and since $R>R_1$, $s\operatorname{Im}\sigma_a$ lies in the compact component (see Figure 2). By considering the MCF starting at $s\Sigma_{\sigma_a}$ as in (i), we again arrive at a contradiction. 
			\end{itemize}
			
			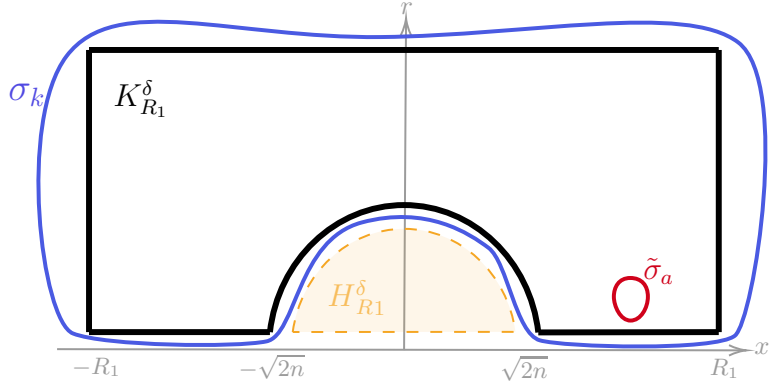
\begin{figure}[ht]
				\centering
				\tikzset{every picture/.style={line width=0.75pt}} 
				
				\begin{tikzpicture}[x=0.60pt,y=0.60pt,yscale=-1,xscale=1]
					
					\draw [color={rgb, 255:red, 155; green, 155; blue, 155 }  ,draw opacity=1 ]   (110.4,217.91) -- (542.94,218.67) ;
					\draw [shift={(544.94,218.68)}, rotate = 180.1] [color={rgb, 255:red, 155; green, 155; blue, 155 }  ,draw opacity=1 ][line width=0.75]    (10.93,-3.29) .. controls (6.95,-1.4) and (3.31,-0.3) .. (0,0) .. controls (3.31,0.3) and (6.95,1.4) .. (10.93,3.29)   ;
					\draw [color={rgb, 255:red, 155; green, 155; blue, 155 }  ,draw opacity=1 ]   (329.16,218.5) -- (330.39,11) ;
					\draw [shift={(330.4,9)}, rotate = 90.34] [color={rgb, 255:red, 155; green, 155; blue, 155 }  ,draw opacity=1 ][line width=0.75]    (10.93,-3.29) .. controls (6.95,-1.4) and (3.31,-0.3) .. (0,0) .. controls (3.31,0.3) and (6.95,1.4) .. (10.93,3.29)   ;
					\draw  [color={rgb, 255:red, 208; green, 2; blue, 27 }  ,draw opacity=1 ][line width=1.5]  (461.4,184.8) .. controls (461.4,178.81) and (464.52,173) .. (471.74,173) .. controls (478.96,173) and (482.86,177.72) .. (482.86,184.8) .. controls (482.86,191.88) and (478.37,199.87) .. (471.93,199.87) .. controls (465.5,199.87) and (461.4,190.79) .. (461.4,184.8) -- cycle ;
					\draw  [draw opacity=0][line width=2.25]  (244.58,207.15) .. controls (249.78,161.97) and (285.67,127) .. (329.16,127) .. controls (372.72,127) and (408.65,162.07) .. (413.77,207.35) -- (329.16,218.5) -- cycle ; \draw  [line width=2.25]  (244.58,207.15) .. controls (249.78,161.97) and (285.67,127) .. (329.16,127) .. controls (372.72,127) and (408.65,162.07) .. (413.77,207.35) ;
					\draw [line width=2.25]    (130.54,207.15) -- (130.54,29.15) ;
					\draw [line width=2.25]    (130.54,207.15) -- (244.58,207.15) ;
					\draw [line width=2.25]    (413.77,207.35) -- (527.54,207.15) ;
					\draw [line width=2.25]    (527.54,29.15) -- (527.54,207.15) ;
					\draw [line width=2.25]    (130.54,29.15) -- (527.54,29.15) ;
					
					\draw  [color={rgb, 255:red, 245; green, 166; blue, 35 }  ,draw opacity=1 ][fill={rgb, 255:red, 245; green, 166; blue, 35 }  ,fill opacity=0.1 ][dash pattern={on 4.5pt off 4.5pt}][line width=0.75]  (258.95,207.09) .. controls (264.06,170.25) and (293.55,142) .. (329.16,142) .. controls (364.75,142) and (394.22,170.21) .. (399.37,207) -- cycle ;
					\draw  [color={rgb, 255:red, 74; green, 90; blue, 226 }  ,draw opacity=1 ][line width=1.5]  (116.4,33) .. controls (147.4,-8) and (221.38,21.98) .. (314.4,21) .. controls (407.42,20.02) and (506.4,-1) .. (539.77,24.35) .. controls (573.15,49.69) and (549.77,199.35) .. (537.77,209.35) .. controls (525.77,219.35) and (430.15,216.69) .. (414.77,212.35) .. controls (399.4,208) and (394.4,163) .. (382.4,154) .. controls (370.4,145) and (343.4,127) .. (302.4,138) .. controls (261.4,149) and (261.4,209) .. (242.4,213) .. controls (223.4,217) and (140.4,216) .. (120.4,208) .. controls (100.4,200) and (85.4,74) .. (116.4,33) -- cycle ;
					
					\draw (548.22,212.08) node [anchor=north west][inner sep=0.75pt]  [font=\footnotesize,color={rgb, 255:red, 155; green, 155; blue, 155 }  ,opacity=1 ]  {$x$};
					\draw (325,-2.6) node [anchor=north west][inner sep=0.75pt]  [font=\footnotesize,color={rgb, 255:red, 155; green, 155; blue, 155 }  ,opacity=1 ]  {$r$};
					\draw (388,220.4) node [anchor=north west][inner sep=0.75pt]  [font=\scriptsize,color={rgb, 255:red, 155; green, 155; blue, 155 }  ,opacity=1 ]  {$\sqrt{2n}$};
					\draw (478.12,158.77) node [anchor=north west][inner sep=0.75pt]  [color={rgb, 255:red, 208; green, 2; blue, 27 }  ,opacity=1 ]  {$\tilde{\sigma }_{a}$};
					\draw (144,46.4) node [anchor=north west][inner sep=0.75pt]    {$K_{R_{1}}^{\delta }$};
					\draw (279,172.4) node [anchor=north west][inner sep=0.75pt]  [color={rgb, 255:red, 245; green, 166; blue, 35 }  ,opacity=0.71 ]  {$H{_{R}^{\delta }}_{1}$};
					\draw (520,222.4) node [anchor=north west][inner sep=0.75pt]  [font=\scriptsize,color={rgb, 255:red, 155; green, 155; blue, 155 }  ,opacity=1 ]  {$R_{1}$};
					\draw (118,221.4) node [anchor=north west][inner sep=0.75pt]  [font=\scriptsize,color={rgb, 255:red, 155; green, 155; blue, 155 }  ,opacity=1 ]  {$-R_{1}$};
					\draw (223,219.4) node [anchor=north west][inner sep=0.75pt]  [font=\scriptsize,color={rgb, 255:red, 155; green, 155; blue, 155 }  ,opacity=1 ]  {$-\sqrt{2n}$};
					\draw (78,48.4) node [anchor=north west][inner sep=0.75pt]  [font=\large,color={rgb, 255:red, 81; green, 74; blue, 226 }  ,opacity=1 ]  {$\sigma _{k}$};
				\end{tikzpicture}
				\caption{$\tilde\sigma_a$ is enclosed by $\sigma_k$.}
			\end{figure}

			\tikzset{every picture/.style={line width=0.75pt}} 
			\begin{figure}[ht]
				\centering
				\begin{tikzpicture}[x=0.60pt,y=0.60pt,yscale=-1,xscale=1]
					
					\draw [color={rgb, 255:red, 155; green, 155; blue, 155 }  ,draw opacity=1 ]   (108.4,216.91) -- (540.94,217.67) ;
					\draw [shift={(542.94,217.68)}, rotate = 180.1] [color={rgb, 255:red, 155; green, 155; blue, 155 }  ,draw opacity=1 ][line width=0.75]    (10.93,-3.29) .. controls (6.95,-1.4) and (3.31,-0.3) .. (0,0) .. controls (3.31,0.3) and (6.95,1.4) .. (10.93,3.29)   ;
					\draw [color={rgb, 255:red, 155; green, 155; blue, 155 }  ,draw opacity=1 ]   (327.16,217.5) -- (328.39,10) ;
					\draw [shift={(328.4,8)}, rotate = 90.34] [color={rgb, 255:red, 155; green, 155; blue, 155 }  ,draw opacity=1 ][line width=0.75]    (10.93,-3.29) .. controls (6.95,-1.4) and (3.31,-0.3) .. (0,0) .. controls (3.31,0.3) and (6.95,1.4) .. (10.93,3.29)   ;
					\draw  [color={rgb, 255:red, 208; green, 2; blue, 27 }  ,draw opacity=1 ][line width=1.5]  (317.4,183.8) .. controls (317.4,177.81) and (320.52,172) .. (327.74,172) .. controls (334.96,172) and (338.86,176.72) .. (338.86,183.8) .. controls (338.86,190.88) and (334.37,198.87) .. (327.93,198.87) .. controls (321.5,198.87) and (317.4,189.79) .. (317.4,183.8) -- cycle ;
					\draw  [draw opacity=0][dash pattern={on 4.5pt off 4.5pt}][line width=0.75]  (242.58,206.15) .. controls (247.78,160.97) and (283.67,126) .. (327.16,126) .. controls (370.72,126) and (406.65,161.07) .. (411.77,206.35) -- (327.16,217.5) -- cycle ; \draw  [color={rgb, 255:red, 128; green, 128; blue, 128 }  ,draw opacity=1 ][dash pattern={on 4.5pt off 4.5pt}][line width=0.75]  (242.58,206.15) .. controls (247.78,160.97) and (283.67,126) .. (327.16,126) .. controls (370.72,126) and (406.65,161.07) .. (411.77,206.35) ;
					\draw [color={rgb, 255:red, 128; green, 128; blue, 128 }  ,draw opacity=1 ][line width=0.75]  [dash pattern={on 4.5pt off 4.5pt}]  (128.54,206.15) -- (128.54,28.15) ;
					\draw [color={rgb, 255:red, 128; green, 128; blue, 128 }  ,draw opacity=1 ][line width=0.75]  [dash pattern={on 4.5pt off 4.5pt}]  (128.54,206.15) -- (242.58,206.15) ;
					\draw [color={rgb, 255:red, 128; green, 128; blue, 128 }  ,draw opacity=1 ][line width=0.75]  [dash pattern={on 4.5pt off 4.5pt}]  (411.77,206.35) -- (525.54,206.15) ;
					\draw [color={rgb, 255:red, 128; green, 128; blue, 128 }  ,draw opacity=1 ][line width=0.75]  [dash pattern={on 4.5pt off 4.5pt}]  (525.54,28.15) -- (525.54,206.15) ;
					\draw [color={rgb, 255:red, 128; green, 128; blue, 128 }  ,draw opacity=1 ][line width=0.75]  [dash pattern={on 4.5pt off 4.5pt}]  (128.54,28.15) -- (525.54,28.15) ;
					\draw  [color={rgb, 255:red, 245; green, 166; blue, 35 }  ,draw opacity=1 ][fill={rgb, 255:red, 245; green, 166; blue, 35 }  ,fill opacity=0.1 ][line width=2.25]  (256.95,206.09) .. controls (262.06,169.25) and (291.55,141) .. (327.16,141) .. controls (362.75,141) and (392.22,169.21) .. (397.37,206) -- cycle ;
					\draw  [color={rgb, 255:red, 74; green, 90; blue, 226 }  ,draw opacity=1 ][line width=1.5]  (402.4,207) .. controls (415.4,197) and (383.4,157) .. (371.4,148) .. controls (359.4,139) and (343.4,126) .. (302.4,137) .. controls (261.4,148) and (244.4,199) .. (249.4,207) .. controls (254.4,215) and (389.4,217) .. (402.4,207) -- cycle ;
					
					\draw (546.22,211.08) node [anchor=north west][inner sep=0.75pt]  [font=\footnotesize,color={rgb, 255:red, 155; green, 155; blue, 155 }  ,opacity=1 ]  {$x$};
					\draw (323,-3.6) node [anchor=north west][inner sep=0.75pt]  [font=\footnotesize,color={rgb, 255:red, 155; green, 155; blue, 155 }  ,opacity=1 ]  {$r$};
					\draw (386,219.4) node [anchor=north west][inner sep=0.75pt]  [font=\scriptsize,color={rgb, 255:red, 155; green, 155; blue, 155 }  ,opacity=1 ]  {$\sqrt{2n}$};
					\draw (333.12,158.77) node [anchor=north west][inner sep=0.75pt]  [color={rgb, 255:red, 208; green, 2; blue, 27 }  ,opacity=1 ]  {$s\sigma _{a}$};
					\draw (140,45.4) node [anchor=north west][inner sep=0.75pt]  [color={rgb, 255:red, 128; green, 128; blue, 128 }  ,opacity=1 ]  {$K_{R_{1}}^{\delta }$};
					\draw (277,171.4) node [anchor=north west][inner sep=0.75pt]  [color={rgb, 255:red, 245; green, 166; blue, 35 }  ,opacity=1 ]  {$H{_{R}^{\delta }}_{1}$};
					\draw (518,221.4) node [anchor=north west][inner sep=0.75pt]  [font=\scriptsize,color={rgb, 255:red, 155; green, 155; blue, 155 }  ,opacity=1 ]  {$R_{1}$};
					\draw (116,220.4) node [anchor=north west][inner sep=0.75pt]  [font=\scriptsize,color={rgb, 255:red, 155; green, 155; blue, 155 }  ,opacity=1 ]  {$-R_{1}$};
					\draw (221,218.4) node [anchor=north west][inner sep=0.75pt]  [font=\scriptsize,color={rgb, 255:red, 155; green, 155; blue, 155 }  ,opacity=1 ]  {$-\sqrt{2n}$};
					\draw (244,141.4) node [anchor=north west][inner sep=0.75pt]  [font=\large,color={rgb, 255:red, 81; green, 74; blue, 226 }  ,opacity=1 ]  {$\sigma _{k}$};
				\end{tikzpicture}
				\caption{$s\sigma_a$ is enclosed by $\sigma_k$.}
			\end{figure}
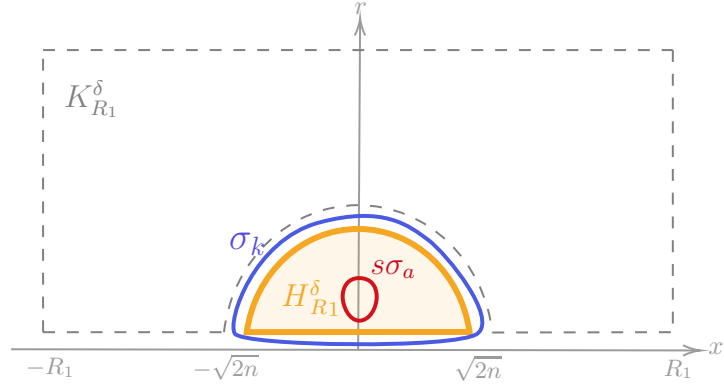
			Hence both cases are impossible, and we conclude that $\operatorname{Im}\sigma_k\cap K_R$ must contain more than one component which intersects $K_{R_1}$ for all $k\ge k_1$. This finishes the proof of Lemma \ref{claim 1}. 
		\end{proof}

		Let $R_1<R_2 < \ldots  < R_k < \ldots$ be any sequence such that $R_k \nearrow +\infty$ as $k\to \infty$. Using Lemma \ref{claim 1} and picking a subsequence of $(\sigma_k)_{k=1}^\infty$ if necessary, we may assume that $\operatorname{Im}\sigma_k\cap K_{R_k}$ has more than one component which intersects $K_{R_1}$. For each $k$, let $N_k\ge 2$ be the number of such connected components. Since there is a uniform positive lower bound on $d_A (K_{R_1}, \mathbb H \setminus K_{R_k})$ for all $k\ge 2$, $N_k$ is uniformly bounded by Theorem \ref{upper bound on length} (see also \eqref{relation between entropy and length}). Taking a further subsequence if necessary, we may assume that $N:=N_k$ is constant.
		
		Thus for each $k$, write $\operatorname{Im}\sigma_k \cap K_{R_k}$ as the disjoint union
		\begin{align} \label{splitting sigma_k into N components}
			\operatorname{Im}\sigma_k \cap K_{R_k} = \operatorname{Im}\sigma^1_k \cup \cdots\cup \operatorname{Im}\sigma^N_k \cup \sigma_k^C,
		\end{align}
		where each $\sigma^i_k$, $i=1, \cdots, N$ is a simple geodesic arc parametrizing a connected component of $\operatorname{Im} \sigma_k \cap K_{R_k}$ which intersects $K_{R_1}$ and $\sigma_k^C$ is the union of any connected components of $\operatorname{Im} \sigma_k \cap K_{R_k}$ which do not intersect $K_{R_1}$. 
		
		For each $k\in \mathbb N$, let $i_k \in \{1, \cdots, N\}$ be arbitrary. Up to reparametrization, there are $c_k>0$ such that
		$$\sigma^{i_k}_k : [-c_k, c_k] \to \mathbb H$$ 
		is a unit speed geodesic. Since the image of $\sigma^{i_k}_k$ intersects $K_{R_1}$, one has 
		$$2c_k \ge d_A(K_{R_1}, \mathbb H \setminus K_{R_2}) >0,$$ 
		and there are  $\bar c_k \in [-c_k, c_k]$ such that $\sigma^{i_k}_k(\bar c_k) \in K_{R_1}$. 
		Taking a further subsequence if necessary, by compactness of $K_{R_1}$, we may assume that 
		$$\sigma^{i_k}_k(\bar c_k) \to q, \ \ \ (\sigma^{i_k}_k)'(\bar c_k) \to w$$ 
		as $k\to \infty$. Let $\tilde \sigma_\infty : J \to \mathbb H$ be the unique complete maximal geodesic with $ \tilde \sigma_\infty (0) = q$ and $\tilde\sigma_\infty'(0)=w$. 
		
		As in the construction of $\sigma_\infty$, $(\sigma^{i_k}_k)$ converges smoothly to $\tilde\sigma_\infty$ in $K_R$ for all $R>R_1$ and $\tilde\sigma_\infty$ is an embedded geodesic in $(\mathbb H, g_A)$. 
		
		\begin{lem} \label{claim 2}
			Up to reparametrization, $\tilde\sigma_\infty = \sigma_\infty$. 
		\end{lem}
		
		\begin{proof} [Proof of Lemma \ref{claim 2}] We may assume that $\sigma^{i_k}_k \neq \sigma_k |_{I_{k, R_k}}$ for all large $k$, since otherwise we have $\tilde\sigma_\infty =\sigma_\infty$ up to reparametrization. This implies that $\operatorname{Im}(\sigma_k|_{I_{k,R_k}})$ and $\operatorname{Im}\sigma_k^{i_k}$ have empty intersection since both of them are connected components of $\operatorname{Im} \sigma_k \cap K_{R_k}$. Assume the contrary, that $\tilde\sigma_\infty \neq \sigma_\infty$. By Lemma \ref{all tori intersect} and that $\tilde\sigma_\infty$, $\sigma_\infty$ are both complete geodesics, they must intersect transversally. Assume that the intersection is in $K_R$ for some $R>R_1$. Since $\sigma_k |_{I_{k, R}}$ converges locally smoothly to $\sigma_\infty$ in $K_{R}$, $\sigma_{k} (I_{k, R})$ also intersects $\tilde\sigma_\infty$ for $k$ large enough. On the other hand, $\sigma^{i_k}_k$ converges smoothly to $\tilde\sigma_\infty$, thus $\sigma^{i_k}_k$ also intersects $\sigma_k |_{I_{k, R}}$ in $K_R$ for large $k$. This contradicts the assumption on $\sigma^{i_k}_k$ and hence the lemma is proved. 
		\end{proof}
		
		Thus the union of subarcs $\operatorname{Im} \sigma_k^1 \cup \cdots  \cup \operatorname{Im} \sigma_k^N$ of $\operatorname{Im}\sigma_k$ converge as $k\to \infty$ to $\sigma_\infty$ locally smoothly with multiplicity $N$. While $\sigma^C_k$ defined in (\ref{splitting sigma_k into N components}) might not be empty, we can show that (by passing to a subsequence if necessary) it also stays close to $\sigma_\infty$.
		
		\begin{lem} \label{sigma^c_k stays close to sigma_infty}
			By passing to a subsequence of $(\sigma_k)$ if necessary,  
			\begin{equation} \label{eqn sigma^C_k closed to sigma_infty}
				|\sqrt{x^2+r^2} - \sqrt{2n}|<R_k^{-1}
			\end{equation}
			for all $(x, r)\in \operatorname{Im} (\sigma_k) \cap K_{R_k}$ and for all $k\in \mathbb N$.
		\end{lem}
		
		\begin{proof}[Proof of Lemma \ref{sigma^c_k stays close to sigma_infty}]
			Let $k\in \mathbb N$ be fixed. First we show that there is $n_k$ so that (\ref{eqn sigma^C_k closed to sigma_infty}) holds for all $(x, r)\in \operatorname{Im} (\sigma_i) \cap K_{R_k}$ and for all $i\ge n_k$. To see this we argue by contradiction: if not, then there is a subsequence of $(\sigma_{k_j})$ of $(\sigma_k)$ and $(x_j, r_j)\in \operatorname{Im} (\sigma_{k_j})\cap K_{R_k}$ so that $|\sqrt{x_j^2+r_j^2} - \sqrt{2n}|\ge R_k^{-1}$ for all $j$. Since $K_{R_k}$ is compact, we may assume that $(x_j, r_j)\to (x_\infty, r_\infty) \in K_{R_k}$, and thus there is a sequence $(\tilde\sigma_j)$ of subarcs of $(\sigma_{k_j})$ which converges locally smoothly to an embedded complete geodesic in $(\mathbb H, g_A)$, which is different from $\sigma_\infty$ since $(x_\infty, r_\infty) \notin \operatorname{Im}(\sigma_\infty)$. Arguing similarly as in the proof of Lemma \ref{claim 2}, this is impossible. Lastly, the lemma is proved by passing to a subsequence of $(\sigma_k)$.
		\end{proof}
		
		A priori, as $k\to \infty$ some portions of $\operatorname{Im} \sigma_k$ might escape to infinity (as $x^2 + r^2\to \infty$) or collapse to the rotational axis $r=  0$. The next lemma shows that this is not the case. 
		
		\begin{lem} \label{claim 3} 
			The sequence of self-shrinkers $(\Sigma_{\sigma_k})_{k=1}^\infty$ converges in Hausdorff distance to the round sphere $\Sigma_{\sigma_\infty}$ as $k\to \infty$. 
		\end{lem}
		
		\begin{proof} [Proof of Lemma \ref{claim 3}] We use a maximum principle argument similar to that in the proof of Lemma \ref{claim 1}. Let $\epsilon>0$. Let $t_0 \in (-1, 0)$ be given by 
			\begin{equation} \label{eqn choice of t_0}
				\sqrt{-t_0} = \frac{\sqrt{2n} + 0.5\epsilon}{\sqrt{2n}+\epsilon}.
			\end{equation}
			Let $k_\epsilon \in \mathbb N$ be large such that $R_{k_\epsilon}^{-1}<\epsilon/2$ and the following holds: there are two horizontal translations and scalings of the Angenent torus $\Sigma_{\bar a_\pm}$ so that
			
			\begin{itemize}
				\item the image of $\bar a_+$ and $\bar a_-$ lie in $K_{R_{k_\epsilon}}$, 
				\item $\sqrt{2n}+R_{k_\epsilon}^{-1}< \sqrt{x^2+r^2} <\sqrt{2n}+\epsilon$ for all $(x, r) \in \operatorname{Im} \bar a_+$, 
				\item $\sqrt{2n}-\epsilon <\sqrt{x^2+r^2}<\sqrt{2n}-R_{k_\epsilon}^{-1}$ for all $(x, r) \in \operatorname{Im} \bar a_-$, and
				\item the MCF starting at $\Sigma_{\bar a_\pm}$ at $t=-1$ shrinks to $(\sqrt{2n} \pm 0.5\epsilon, 0)$ at time $t_e <t_0$. 
			\end{itemize}
			By Lemma \ref{sigma^c_k stays close to sigma_infty}, $\bar a_{\pm}$ do not intersect with $\sigma_k$ when $k\ge k_\epsilon$.
			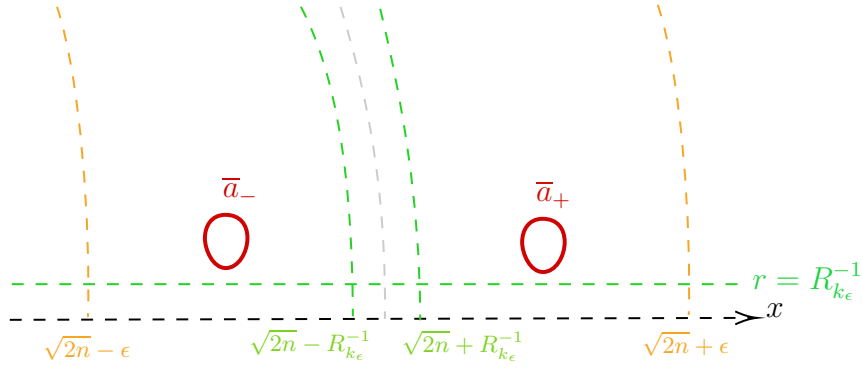
\begin{figure}[ht]
				\tikzset{every picture/.style={line width=0.75pt}} 
				
				\begin{tikzpicture}[x=0.75pt,y=0.75pt,yscale=-1,xscale=1]
					
					\draw  [dash pattern={on 4.5pt off 4.5pt}]  (124,161) -- (498.8,160.01) ;
					\draw [shift={(500.8,160)}, rotate = 179.85] [color={rgb, 255:red, 0; green, 0; blue, 0 }  ][line width=0.75]    (10.93,-3.29) .. controls (6.95,-1.4) and (3.31,-0.3) .. (0,0) .. controls (3.31,0.3) and (6.95,1.4) .. (10.93,3.29)   ;
					\draw [color={rgb, 255:red, 155; green, 155; blue, 155 }  ,draw opacity=0.52 ] [dash pattern={on 4.5pt off 4.5pt}]  (290.8,3) .. controls (306.8,55) and (313.8,103) .. (313.19,160.5) ;
					\draw [color={rgb, 255:red, 245; green, 166; blue, 35 }  ,draw opacity=1 ] [dash pattern={on 4.5pt off 4.5pt}]  (147.8,3) .. controls (159.8,39) and (164.8,115) .. (163.4,159.5) ;
					\draw [color={rgb, 255:red, 35; green, 211; blue, 33 }  ,draw opacity=1 ] [dash pattern={on 4.5pt off 4.5pt}]  (125,143) -- (495.8,143) ;
					\draw [color={rgb, 255:red, 35; green, 211; blue, 33 }  ,draw opacity=1 ] [dash pattern={on 4.5pt off 4.5pt}]  (270.8,3) .. controls (288.8,37) and (296.8,89) .. (297.01,159.22) ;
					\draw  [color={rgb, 255:red, 208; green, 2; blue, 2 }  ,draw opacity=1 ][line width=1.5]  (382.4,121.8) .. controls (382.4,115.81) and (385.52,110) .. (392.74,110) .. controls (399.96,110) and (403.86,114.72) .. (403.86,121.8) .. controls (403.86,128.88) and (399.37,136.87) .. (392.93,136.87) .. controls (386.5,136.87) and (382.4,127.79) .. (382.4,121.8) -- cycle ;
					\draw  [color={rgb, 255:red, 208; green, 2; blue, 2 }  ,draw opacity=1 ][line width=1.5]  (222.4,119.8) .. controls (222.4,113.81) and (225.52,108) .. (232.74,108) .. controls (239.96,108) and (243.86,112.72) .. (243.86,119.8) .. controls (243.86,126.88) and (239.37,134.87) .. (232.93,134.87) .. controls (226.5,134.87) and (222.4,125.79) .. (222.4,119.8) -- cycle ;
					\draw [color={rgb, 255:red, 35; green, 211; blue, 33 }  ,draw opacity=1 ] [dash pattern={on 4.5pt off 4.5pt}]  (310.8,4) .. controls (321.8,50) and (330.8,100) .. (331.01,161.22) ;
					\draw [color={rgb, 255:red, 245; green, 166; blue, 35 }  ,draw opacity=1 ] [dash pattern={on 4.5pt off 4.5pt}]  (450.8,2) .. controls (462.8,38) and (467.8,114) .. (466.4,158.5) ;
					
					\draw (320.4,164.9) node [anchor=north west][inner sep=0.75pt]  [font=\scriptsize,color={rgb, 255:red, 126; green, 211; blue, 33 }  ,opacity=1 ]  {$\sqrt{2n} +R_{k_{\epsilon }}^{-1}$};
					\draw (243.4,164.4) node [anchor=north west][inner sep=0.75pt]  [font=\scriptsize,color={rgb, 255:red, 126; green, 211; blue, 33 }  ,opacity=1 ]  {$\sqrt{2n} -R_{k_{\epsilon }}^{-1}$};
					\draw (139.4,167.4) node [anchor=north west][inner sep=0.75pt]  [font=\scriptsize,color={rgb, 255:red, 245; green, 166; blue, 35 }  ,opacity=1 ]  {$\sqrt{2n} -\epsilon $};
					\draw (441.4,165.4) node [anchor=north west][inner sep=0.75pt]  [font=\scriptsize,color={rgb, 255:red, 245; green, 166; blue, 35 }  ,opacity=1 ]  {$\sqrt{2n} +\epsilon $};
					\draw (504,151.4) node [anchor=north west][inner sep=0.75pt]    {$x$};
					\draw (497,130.4) node [anchor=north west][inner sep=0.75pt]  [color={rgb, 255:red, 33; green, 211; blue, 65 }  ,opacity=1 ]  {$r=R_{k_{\epsilon }}^{-1}$};
					\draw (388,89.4) node [anchor=north west][inner sep=0.75pt]  [color={rgb, 255:red, 208; green, 13; blue, 2 }  ,opacity=1 ]  {$\overline{a}_{+}$};
					\draw (230,87.4) node [anchor=north west][inner sep=0.75pt]  [color={rgb, 255:red, 208; green, 2; blue, 2 }  ,opacity=1 ]  {$\overline{a}_{-}$};

				\end{tikzpicture}
				\caption{Choices of $\bar a_\pm$.}
			\end{figure}
			
			Now we claim that 
			\begin{align} \label{Hausdorff converges to round sphere eqn}
				\sqrt{2n}-\epsilon<\sqrt{x^2+r^2}<\sqrt{2n}+\epsilon, \ \ \text{ for all }(x,r)\in \operatorname{Im}\sigma_k, \ k\ge k_\epsilon.
			\end{align}
			By Lemma \ref{sigma^c_k stays close to sigma_infty} and that $R_{k_\epsilon}^{-1}<\epsilon$, it suffices to consider points outside $K_{R_{k_\epsilon}}$. Fix any $k\ge k_\epsilon$. If there is $(x,r)\in \operatorname{Im}(\sigma_k)$ so that $\sqrt{x^2+r^2} \ge \sqrt{2n} +\epsilon$, then since each $\sigma_k$ is connected, there is a subarc $\beta$ of $\sigma_k$ connecting $(x,r)$ and $\operatorname{Im} (\sigma_k) \cap K_{R_{k_\epsilon}}$ passing through the region 
			\[ [\sqrt{2n}, \sqrt{2n} +\epsilon] \times (0,R_{k_\epsilon}].\] 
			But this is impossible: by (\ref{eqn choice of t_0}), for all $-1<t<t_0$, $\sqrt{-t}\beta$ contains a point $\{\sqrt{2n} + 0.5\epsilon\} \times (0,R_{k_\epsilon})$. Hence the MCF starting at $\Sigma_{\bar a_+}$ (at time $-1$) would intersect with $\sqrt{-t}\beta$ for some $t\in (-1, t_e)$ and this contradicts the parabolic maximum principle, since $\Sigma_{\sigma_k}$ and $\Sigma_{\bar a_+}$ are disjoint. Using $\Sigma_{\bar a_-}$ and arguing similarly, we conclude that $\sqrt{x^2+r^2}\le \sqrt{2n}-\epsilon$ is also impossible. 
			\begin{figure}[ht]
				\tikzset{every picture/.style={line width=0.75pt}} 
				
				\begin{tikzpicture}[x=0.75pt,y=0.75pt,yscale=-1,xscale=1]
					
					\draw  [dash pattern={on 4.5pt off 4.5pt}]  (8,139.51) -- (312.47,138.67) ;
					\draw [shift={(314.47,138.66)}, rotate = 179.84] [color={rgb, 255:red, 0; green, 0; blue, 0 }  ][line width=0.75]    (10.93,-3.29) .. controls (6.95,-1.4) and (3.31,-0.3) .. (0,0) .. controls (3.31,0.3) and (6.95,1.4) .. (10.93,3.29)   ;
					\draw [color={rgb, 255:red, 155; green, 155; blue, 155 }  ,draw opacity=0.52 ] [dash pattern={on 4.5pt off 4.5pt}]  (143.67,5.85) .. controls (156.68,49.84) and (162.37,90.44) .. (161.88,139.08) ;
					\draw [color={rgb, 255:red, 245; green, 166; blue, 35 }  ,draw opacity=1 ] [dash pattern={on 4.5pt off 4.5pt}]  (27.36,5.85) .. controls (37.12,36.3) and (41.18,100.59) .. (40.05,138.24) ;
					\draw [color={rgb, 255:red, 35; green, 211; blue, 33 }  ,draw opacity=1 ] [dash pattern={on 4.5pt off 4.5pt}]  (8.81,124.28) -- (310.41,124.28) ;
					\draw [color={rgb, 255:red, 35; green, 211; blue, 33 }  ,draw opacity=1 ] [dash pattern={on 4.5pt off 4.5pt}]  (127.4,5.85) .. controls (142.04,34.61) and (148.55,78.6) .. (148.72,138) ;
					\draw  [color={rgb, 255:red, 208; green, 2; blue, 2 }  ,draw opacity=1 ][line width=1.5]  (218.17,106.35) .. controls (218.17,101.28) and (220.71,96.36) .. (226.58,96.36) .. controls (232.45,96.36) and (235.62,100.36) .. (235.62,106.35) .. controls (235.62,112.34) and (231.97,119.09) .. (226.74,119.09) .. controls (221.5,119.09) and (218.17,111.41) .. (218.17,106.35) -- cycle ;
					\draw [color={rgb, 255:red, 35; green, 211; blue, 33 }  ,draw opacity=1 ] [dash pattern={on 4.5pt off 4.5pt}]  (159.93,6.69) .. controls (168.88,45.61) and (176.2,87.9) .. (176.37,139.69) ;
					\draw [color={rgb, 255:red, 245; green, 166; blue, 35 }  ,draw opacity=1 ] [dash pattern={on 4.5pt off 4.5pt}]  (273.8,5) .. controls (283.56,35.45) and (287.63,99.75) .. (286.49,137.39) ;
					\draw  [dash pattern={on 4.5pt off 4.5pt}]  (337,139.51) -- (641.47,138.67) ;
					\draw [shift={(643.47,138.66)}, rotate = 179.84] [color={rgb, 255:red, 0; green, 0; blue, 0 }  ][line width=0.75]    (10.93,-3.29) .. controls (6.95,-1.4) and (3.31,-0.3) .. (0,0) .. controls (3.31,0.3) and (6.95,1.4) .. (10.93,3.29)   ;
					\draw [color={rgb, 255:red, 155; green, 155; blue, 155 }  ,draw opacity=0.52 ] [dash pattern={on 4.5pt off 4.5pt}]  (472.67,5.85) .. controls (485.68,49.84) and (491.37,90.44) .. (490.88,139.08) ;
					\draw [color={rgb, 255:red, 245; green, 166; blue, 35 }  ,draw opacity=1 ] [dash pattern={on 4.5pt off 4.5pt}]  (356.36,5.85) .. controls (366.12,36.3) and (370.18,100.59) .. (369.05,138.24) ;
					\draw [color={rgb, 255:red, 35; green, 211; blue, 33 }  ,draw opacity=1 ] [dash pattern={on 4.5pt off 4.5pt}]  (337.81,124.28) -- (639.41,124.28) ;
					\draw [color={rgb, 255:red, 35; green, 211; blue, 33 }  ,draw opacity=1 ] [dash pattern={on 4.5pt off 4.5pt}]  (456.4,5.85) .. controls (471.04,34.61) and (477.55,78.6) .. (477.72,138) ;
					\draw  [color={rgb, 255:red, 208; green, 2; blue, 2 }  ,draw opacity=1 ][line width=1.5]  (417.03,104.65) .. controls (417.03,99.59) and (419.57,94.67) .. (425.44,94.67) .. controls (431.31,94.67) and (434.49,98.66) .. (434.49,104.65) .. controls (434.49,110.64) and (430.84,117.4) .. (425.6,117.4) .. controls (420.37,117.4) and (417.03,109.72) .. (417.03,104.65) -- cycle ;
					\draw [color={rgb, 255:red, 35; green, 211; blue, 33 }  ,draw opacity=1 ] [dash pattern={on 4.5pt off 4.5pt}]  (488.93,6.69) .. controls (497.88,45.61) and (505.2,87.9) .. (505.37,139.69) ;
					\draw [color={rgb, 255:red, 245; green, 166; blue, 35 }  ,draw opacity=1 ] [dash pattern={on 4.5pt off 4.5pt}]  (602.8,5) .. controls (612.56,35.45) and (616.63,99.75) .. (615.49,137.39) ;
					\draw [color={rgb, 255:red, 74; green, 144; blue, 226 }  ,draw opacity=1 ][line width=2.25]    (135,8) .. controls (166.8,55) and (148.8,127) .. (167.8,131) .. controls (186.8,135) and (263.8,133) .. (285.8,132) ;
					\draw [color={rgb, 255:red, 74; green, 144; blue, 226 }  ,draw opacity=1 ][line width=2.25]    (466,4) .. controls (491.8,50) and (501.8,111) .. (498.8,127) .. controls (495.8,143) and (410.8,121) .. (369.8,132) ;
					\draw [color={rgb, 255:red, 74; green, 144; blue, 226 }  ,draw opacity=1 ][line width=2.25]  [dash pattern={on 2.53pt off 3.02pt}]  (285.8,132) .. controls (296.8,132) and (300.8,131) .. (310.8,132) ;
					\draw [color={rgb, 255:red, 74; green, 144; blue, 226 }  ,draw opacity=1 ][line width=2.25]  [dash pattern={on 2.53pt off 3.02pt}]  (344.8,132) .. controls (354.8,134) and (361.8,134) .. (369.8,132) ;
					
					\draw (220.67,77.46) node [anchor=north west][inner sep=0.75pt]  [color={rgb, 255:red, 208; green, 13; blue, 2 }  ,opacity=1 ]  {$\overline{a}_{+}$};
					\draw (421.26,75.77) node [anchor=north west][inner sep=0.75pt]  [color={rgb, 255:red, 208; green, 2; blue, 2 }  ,opacity=1 ]  {$\overline{a}_{-}$};
					\draw (263.4,142.4) node [anchor=north west][inner sep=0.75pt]  [font=\scriptsize,color={rgb, 255:red, 245; green, 166; blue, 35 }  ,opacity=1 ]  {$\sqrt{2n} +\epsilon $};
					\draw (347.4,142.4) node [anchor=north west][inner sep=0.75pt]  [font=\scriptsize,color={rgb, 255:red, 245; green, 166; blue, 35 }  ,opacity=1 ]  {$\sqrt{2n} -\epsilon $};
					\draw (122,17.4) node [anchor=north west][inner sep=0.75pt]  [font=\large,color={rgb, 255:red, 74; green, 144; blue, 226 }  ,opacity=1 ]  {$\sigma _{k}$};
					\draw (475,-0.6) node [anchor=north west][inner sep=0.75pt]  [font=\large,color={rgb, 255:red, 74; green, 144; blue, 226 }  ,opacity=1 ]  {$\sigma _{k}$};

				\end{tikzpicture}
				\caption{A subarc $\beta$ of $\sigma_k$ passing through $\{ x = \sqrt{2n} + \epsilon\}$ (left) or $\{ x = \sqrt{2n} - \epsilon\}$ (right).}
			\end{figure}
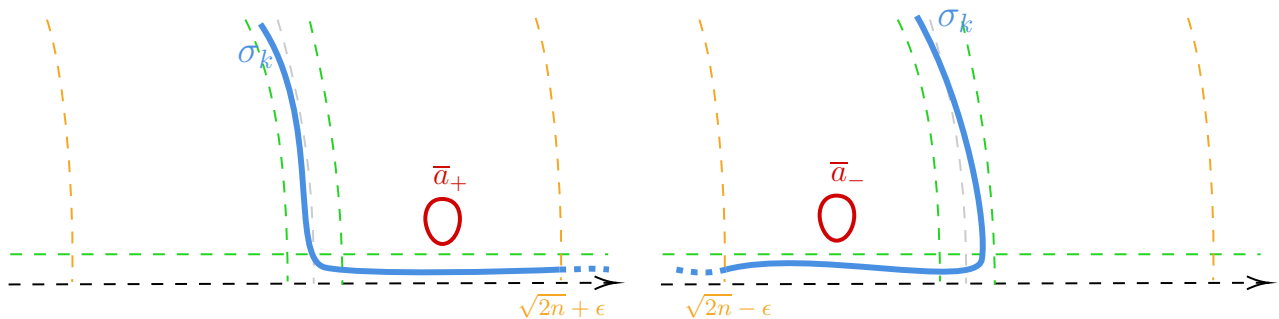

			Thus (\ref{Hausdorff converges to round sphere eqn}) is shown and this implies $d_0 (p, \Sigma_{\sigma_\infty}) < \epsilon$ for all $k \ge k_\epsilon$ and $p\in \Sigma_{\sigma_k}$. On the other hand, since $(\sigma _k|_{I_{k,R_k}})$ converges smoothly to $\sigma_\infty$ in $K_{2/\epsilon}$, by choosing a larger $k_\epsilon$ if necessary, we may assume that $d_0 (p, \Sigma_{\sigma_k}) < \epsilon/2$ for all $p=(x, r\omega)\in \Sigma_{\sigma_\infty}$ with $(x,r) \in K_{2/\epsilon}$, $k\ge k_\epsilon$ and $\omega \in \mathbb{S}^{n-1}$. Thus $d_0 (p, \Sigma_{\sigma_k}) < \epsilon$ for all $p\in \Sigma_{\sigma_\infty}$ and $k\ge k_\epsilon$. By definition of the Hausdorff distance $d^{\mathcal H}$, 
			$$ d^{\mathcal H} (\Sigma_{\sigma_\infty}, \Sigma_{\sigma_k}) = \max\{ \sup_{p\in \Sigma_{\sigma_k}} d_0 (p, \Sigma_{\sigma_\infty}), \sup_{p\in \Sigma_{\sigma_\infty}} d_0 (p, \Sigma_{\sigma_k})\} <\epsilon$$
			for all $k\ge k_\epsilon$. Since $\epsilon>0$ is arbitrary, $(\Sigma_{\sigma_k})_{k=1}^\infty$ converges in Hausdorff distance to the round sphere $ \Sigma_{\sigma_\infty}$ and this finishes the proof of Lemma \ref{claim 3}.  
		\end{proof}
		
		To summarize, we have shown that the sequence of self-shrinkers $(\Sigma_{\sigma_k})$ converges in Hausdorff distance to the sphere $\Sigma_{\sigma_\infty}$. Moreover, for all $\epsilon >0$, the set 
		$$\Sigma_{\sigma_k} \setminus (B_{(\sqrt{2n},0)}(\epsilon ) \cup  (B_{(-\sqrt{2n},0)}(\epsilon ))$$ decomposes into $N$ disjoint graphs $\Sigma^1_k,\cdots \Sigma^N_k$ over $\Sigma_{\sigma_\infty}$ for large enough $k$, where $N\ge 2$, and for each $i=1, \cdots, N$, the convergence 
		$$\Sigma^i_k \to \Sigma_{\sigma_\infty}\setminus (B_{(\sqrt{2n},0)}(\epsilon ) \cup  (B_{(-\sqrt{2n},0)}(\epsilon )), \ \ \text{ as } k\to \infty,$$
		is smooth graphical convergence.
		
		Thus we can apply \cite[Proposition 3.2]{CM2} to conclude that $\Sigma_{\sigma_\infty}$ is $L$-stable. Although \cite[Proposition 3.2]{CM2} is stated only for $n=2$, the same proof, which we now briefly describe, works for all $n\ge 2$ with only notational changes. 
		
		For each fixed $k$, since $\Sigma_k^1, \cdots, \Sigma_k^N$ are disjoint, we can order these $N$ sheets by height (with respect to the outward unit normal of $\Sigma_{\sigma_\infty}= \mathcal S$). Let the top and the bottom layers be represented respectively by two functions $w_k^+$, $w_k^-$ defined on $\Omega_k \subset \mathcal S\setminus \{(\pm \sqrt{2n},0)\}$ so that $w_k^+ > w_k^-$, $\Omega_k \subset \Omega_{k+1}$ and $\cup_k \Omega_k = \mathcal S \setminus \{(\pm \sqrt{2n},0)\}$. Fixing $x_0\in \mathcal S$, then the sequence of functions 
		$$ u_k = \frac{w^+_k - w_k^-}{w_k^+(x_0) - w_k^-(x_0)}$$
		converges locally smoothly on $\mathcal S \setminus \{(\pm \sqrt{2n},0)\}$ to a smooth function $u$ which satisfies $Lu=0$, where $L$ is the stability operator (\ref{L operator definition}). Using the Harnack inequality for linear second order elliptic equations and a maximum principle for minimal hypersurfaces (this is where Lemma \ref{claim 3} is used), one can bound $u$ uniformly. Hence $u$ extends across $\{(\pm \sqrt{2n},0)\}$ by the removable singularities lemma for $L$-harmonic functions \cite{Serrin}. Thus we have constructed a positive function on $\mathcal S$ which satisfies $Lu=0$ and this is sufficient to conclude that $\mathcal S$ is $L$-stable (see \cite{CM2} for more details and \cite{FCS} for a general statement). Since all properly embedded self-shrinkers in $\mathbb R^{n+1}$ are $L$-unstable (\cite{CM1}, see also \cite[Theorem 0.5]{CM2}), we have arrived at a contradiction and hence $\Sigma_{\sigma_\infty}$ is not the sphere $\mathcal S$. 
		
		Next we argue by contradiction that $\Sigma_{\sigma_\infty}$ is also not the plane $\mathcal P$ nor the cylinder $\mathcal C$. The arguments are similar to those for $\mathcal S$, thus we just point out the differences. 
		
		If $\Sigma_{\sigma_\infty}$ is the plane $\mathcal P$, then $\sigma_k |_{I_{k, R}}$ converges smoothly to the $r$-axis in $K_R$. Hence there is $k_1= k_1(R)\in \mathbb N$ so that $\sqrt{x^2+r^2} < R^{-1}$ for all $(r, x)\in \operatorname{Im} \sigma_k \cap K_{R}$ and $k\ge k_1$. For any $R \ge R_1$, let 
		$$K^\pm_{R} := \{ (x,r) \in K_{R} : \pm x \ge R^{-1}\}$$
		One can argue that either one of $K^\pm_{R_1} $ must intersect the image of $\sigma_k$ when $k\ge k_1$: if not, then either one of $K^\pm_{R_1} $ would lie in the compact region bounded by $\sigma_k$. This would lead to a contradiction by putting in suitably scaled and horizontally translated Angenent doughnuts in $K^\pm_{R_1}$. 
		
		Similar to the previous argument for the sphere $\mathcal S$, there is $N\ge 2$ so that for all $k\in \mathbb N$, $\operatorname{Im} \sigma_k\cap K_{R_k}$ contains $N$ connected components which intersect $K_{R_1}$, and $\operatorname{Im}\sigma_k\cap K_{R_k}$ converges smoothly graphically to the $r$-axis with multiplicity $N$. 
		
		As in the proof of Lemma \ref{claim 3}, for all $R>R_1$, one can show that $\operatorname{Im} \sigma_k \cap K_R$ converges (locally) in Hausdorff distance to $\operatorname{Im} \sigma_\infty\cap K_R$ as $k\to \infty$ (unlike the case of the sphere, there might be mass loss as $R\to \infty$). This is still sufficient for us to apply \cite[Proposition 3.2]{CM2} to conclude that the plane is $L$-stable, which is impossible \cite{CM1}.
		
		For the remaining case for the cylinder $\mathcal C$, the curves $\sigma_k |_{I_{k, R}}$ converge smoothly to $\{r = \sqrt{2(n-1)}\}$ in $K_R$. Hence there is $k_1= k_1(R)\in \mathbb N$ such that $|r- \sqrt{2(n-1)}| < R^{-1}$ for all $(x,r)\in \operatorname{Im} \sigma_k \cap K_{R}$ and $k\ge k_1$. By fitting Angenent doughnuts inside 
		\begin{align*}
			K^>_{R_1} &= \{ (x,r) \in K_{R_1} : r >\sqrt{2(n-1)} +R^{-1}\}, \\
			K^<_{R_1} &= \{ (x,r) \in K_{R_1} : r <\sqrt{2(n-1)} -R^{-1}\}
		\end{align*}
		respectively (in $K^>_{R_1}$ we insert a large Angenent doughnut), there is $N\ge 2$ so that $\operatorname{Im}\sigma_k\cap K_{R_k}$ converges smoothly graphically to $\{ r= \sqrt{2(n-1)}\}$ with multiplicity $N$. Then again we apply \cite[Proposition 3.2]{CM2}.
		
		To sum up, $\Sigma_{\sigma_\infty}$ is neither the plane $\mathcal P$ nor the cylinder $\mathcal C$ nor the sphere $\mathcal S$. By the classification Theorem \ref{partial classification of KM}, $\sigma_\infty$ is an embedded closed geodesic in $(\mathbb H, g_A)$ and the convergence $\Sigma_{\sigma_k} \to \Sigma_{\sigma_\infty}$ is smooth. This finishes the proof of Theorem \ref{smooth compactness}.
	\end{proof}
	
	\begin{rem}
		In the proof of Theorem \ref{smooth compactness}, we argued by contradiction using $L$-stability that $\Sigma_{\sigma_{\infty}}$ is neither the sphere $\mathcal S$, nor the cylinder $\mathcal C$ nor the plane $\mathcal P$. In this remark we give an alternative argument ruling out $\mathcal S$ and $\mathcal C$ by instead using the entropy bound $E_n \le E_2$ (see Lemma \ref{properties of E_n lemma}). By Lemma \ref{claim 1} and Lemma \ref{claim 2}, for each $k$ large, one can find $N$ disjoint subarcs $\sigma_k^1 \cdots, \sigma_k^N$ of $\sigma_k $ so that for each $i=1, \cdots, N$, $(\sigma_k ^i)$ converges locally smoothly to $\sigma_\infty$ as $k\to \infty$. In particular, 
		$$ \lim_{k\to \infty} L_A (\sigma_k) \ge N L_A(\sigma_\infty).$$
		By (\ref{relation between entropy and length}), we obtain 
		$$ E_2 \ge E_n \ge \lim_{k\to \infty} \lambda (\Sigma_{\sigma_k}) \ge N \lambda (\Sigma_{\sigma_\infty}). $$
		If $\Sigma_{\sigma_\infty} = \mathcal S$ or $\mathcal C$, then $\lambda (\Sigma_\infty) > \sqrt 2$ by \cite[A.4 Lemma]{Stone}. Thus $N\sqrt 2 \le E_2$, which is impossible as $N \ge 2$ and $E_2 \sim 2.2476< 2\sqrt 2$. 
		
		Since $E_n >2$ and $\lambda (\mathcal P)=1$, this argument is however not sufficient to show $\Sigma_{\sigma_\infty} \neq \mathcal P$. 
	\end{rem}

	Next we prove several consequences of Theorem \ref{smooth compactness}.
	
	First, a simple contradiction argument using Theorem \ref{smooth compactness} shows that the profile curve of every embedded rotationally symmetric self-shrinking doughnut must lie in a fixed compact subset in $\mathbb H$. Together with Theorem \ref{upper bound on length} and \cite[Theorem 1.1]{yakov2}, we obtain an upper bound on the index. (for the definition of index of a self-shrinker, see \cite[Definition 2.18]{yakov2}.) 
	
	\begin{cor} \label{cor index upper bound}
		There is $I\in \mathbb N$ so that for any complete embedded rotationally symmetric self-shrinker $\Sigma$ in $\mathbb R^3$, one has $i(\Sigma)\le I$, where $i(\Sigma)$ is the index of $\Sigma$. 
	\end{cor}
	
	Unlike our Theorem \ref{upper bound on length}, the bound $I$ in the above corollary is not explicit. 
	
	The next corollary and Theorem \ref{finitely many symmetric one} generalize \cite[Theorem 1.3]{Alex} to any dimension and by allowing rotationally symmetric self-shrinkers with non-convex profile curve and is a consequence of the fact that the Angenent metrics $g_A$ are real analytic. See also \cite{ChenMa2} for a similar theorem for Lagrangian self-shrinking tori in $\mathbb R^4$.
	
	\begin{cor} \label{finiteness of entropy value}
		For each $n\ge 2$, there is a finite set $S_n \subset [1,E_n]$ so that $\lambda (\Sigma) \in S_n$ for all complete embedded rotationally symmetric self-shrinkers $\Sigma$ in $\mathbb R^{n+1}$.
	\end{cor}
	
	\begin{proof}
		By Theorem \ref{partial classification of KM}, it suffices to consider only rotationally symmetric self-shrinkers that are diffeomorphic to $\mathbb S^{n-1} \times\mathbb S^1$. Let $\sigma$ be the profile curve of a rotationally symmetric self-shrinking doughnut in $\mathbb R^{n+1}$. Thus $\sigma$ is a geodesic in $(\mathbb H, g_A)$. Since $g_A$ is real analytic, the mapping 
		$$\gamma \mapsto E_A(\gamma):= \frac 12\int_{\mathbb S^1} g_A(\dot\gamma, \dot\gamma)$$
		is a real analytic functional defined on all closed curves $\gamma$ close to $\sigma$. Note that $DE_A(\gamma)=0$ if and only if $\gamma$ is a geodesic, here $DE_A(\gamma)$ is the $L^2$-gradient of $E_A$ at $\gamma$. By the celebrated \L ojasiewicz-Simon gradient inequality \cite[equation (2.2)]{Leonsimon}, there is $C_2 >0$ and $\theta \in (0,1/2)$ so that 
		$$| E_A(\gamma) - E_A(\sigma)|^{1-\theta} \le C_2 \| DE_A(\gamma)\|_{L^2},$$
		for all $\gamma$ which are $C^{2,\alpha}$-close to $\sigma$. Hence if $\gamma$ is another geodesic close to $\sigma$, we have $E_A(\gamma) = E_A(\sigma)$. Since geodesics are parametrized by constant length, one has $L_A(\gamma) = L_A(\sigma)$. Thus the length functional $\gamma \mapsto L_A(\gamma)$ is locally constant on the space of all closed geodesics. Since the space of simple closed geodesics is compact by Theorem \ref{smooth compactness}, the length functional has a finite image. Together with (\ref{relation between entropy and length}), this finishes the proof of the corollary. 
	\end{proof}
	
	Lastly, we prove Theorem \ref{finitely many symmetric one}.
	
	\begin{proof}[Proof of Theorem \ref{finitely many symmetric one}] In \cite[section 5]{Alex}, Mramor studies the Poincar\'e map of the (up to renormalization of length) geodesic equation in $(\mathbb H, g_A)$
		\begin{align} \label{geodesic equ of H, g}
			x'=\cos \theta, \ \ r' = \sin \theta, \ \ \theta' = \frac{x}{2} \sin \theta + \left( \frac{n-1}{r} - \frac{r}{2}\right) \cos \theta.
		\end{align}
		For any $R >0$ and $\theta$, let $(x_{R,\theta}(t), y_{R, \theta} (t), \theta_{R, \theta} (t))$ be the maximal solution to (\ref{geodesic equ of H, g}) with initial value $(0, R, \theta)$. Let $T^* = T^* (R, \theta)$ be the second time at which $x_{R, \theta} = 0$ occurs. 
		
		By \cite[Lemma 5.1]{Alex}, the fixed points of the Poincar\'e map 
		$$ P : \mathbb R \times (0,\infty) \to \mathbb R \times (0,\infty), \ \ P(R, \theta) = (r_{R, \theta}(T^*), \theta_{R, \theta}(T^*))$$
		are either isolated points or analytic curves in $\mathbb R \times (0,\infty)$. Arguing as in the proof of \cite[Theorem 1.3]{Alex}, the map 
		$$P_f : (0,\infty) \to (0,\infty),\ \ \  P_f (R) = r_{R,0} (T^*(R,0), 0)$$
		has isolated fixed points. Since the profile curve of any embedded self-shrinker with reflectional symmetry must intersect the fixed points of $P_f$, together with Theorem \ref{smooth compactness}, Theorem \ref{finitely many symmetric one} has been proven. 
	\end{proof}
	
	\section{Appendix: the sequence $(E_n)$}
	In this appendix, we show the following lemma. 
	
	\begin{lem} \label{properties of E_n lemma} 
		The sequence $(E_n)$ defined in (\ref{E_n constants}) satisfies $2<E_n \le E_2$ and
		\begin{align} \label{limit of E_n}
			\lim_{n\to \infty} E_n= \sqrt{\frac{4\pi}{3}}. 
		\end{align}
	\end{lem}
	\begin{proof} 
		It is proved in \cite[A.4 Lemma]{Stone} that the entropy of the $n$-sphere $\lambda (\mathbb{S}^n)$ satisfies 
		\begin{align} \label{entropy of sphere formula}
			\lambda (\mathbb{S}^n) = \left( \frac{n}{2\pi e}\right)^{n/2} \omega_n
		\end{align}
		and the sequence $(\lambda (\mathbb{S}^n) )$ is strictly decreasing. Also,  
		\begin{align} \label{limit of entropy of sphere}
			\lim_{n\to \infty} \lambda (\mathbb{S}^n) = \sqrt 2.
		\end{align}
		From (\ref{E_n constants}) and (\ref{entropy of sphere formula}) we obtain
		\begin{align} \label{formula for E_n in term of entropy of spheres}
			E_n = \sqrt{\frac{2\pi}{3} \frac{1 +x_n}{1 + 2x_n/3}} \left(\frac{1}{e} \left( 1 + x_n \right)^{1/x_n}\right)^{a_n /4} \lambda (\mathbb{S}^{n-1}),
		\end{align}
		where 
		\begin{align*}
			a_n = y_n - 2(n-1), \ \ x_n = \frac{a_n}{2(n-1)}. 
		\end{align*}
		Direct calculations give 
		\begin{align} \label{inequality for a_n and x_n} 
			\frac 12  < a_n <\frac 23 , \ \ 0<x_n <1
		\end{align}
		and 
		\begin{align} \label{limit of a_n and x_n}
			\lim_{n\to \infty} a_n = \frac{2}{3}, \ \ \ \lim_{n\to \infty} x_n = 0.
		\end{align}
		Using (\ref{formula for E_n in term of entropy of spheres}), (\ref{limit of entropy of sphere}) and (\ref{limit of a_n and x_n}), one obtains (\ref{limit of E_n}). 
		
		Next we show $2< E_n \le E_2$. By the Taylor expansion of $\ln (1+x)$, we have $ x + x^3/3 >\ln (1+x) > x - x^2/2$ for all $x\in (0,1)$. Thus
		$$e^{\frac{x^2}{3}} > \frac{1}{e} (1+x)^{1/x} > e^{-\frac{x}{2}}, \ \ \text{ for all } x\in (0,1).$$
		Together with $\lambda (\mathbb{S}^{n-1}) > \sqrt 2$, (\ref{inequality for a_n and x_n}) and (\ref{formula for E_n in term of entropy of spheres}),
		\begin{align} \label{last inequalit of E_n}
			\frac{\sqrt{2\pi(3+(n-1)^{-1})}}{3} e^{\frac{1}{162 (n-1)^2}} \lambda (\mathbb{S}^{n-1}) >  E_n >  \sqrt{\frac{4\pi}{3}}e^{-\frac{1}{36(n-1)}}. 
		\end{align}
		Since $\lambda (\mathbb{S}^{n-1})$ is decreasing, the upper bound in (\ref{last inequalit of E_n}) is strictly decreasing in $n$. Also, the lower bound in (\ref{last inequalit of E_n}) is strictly increasing in $n$. Plugging in $n=4$ in the upper and lower bound of (\ref{last inequalit of E_n}) gives 
		$$2.21823\sim \frac{\sqrt{10} \pi}{e^{1093/729}} > E_n > \sqrt{\frac{4\pi}{3}}e^{-\frac{1}{108}} \sim 2.02780, \ \ \text{ for all } n\ge 4.$$
		The inequality implies $2 <E_n < E_2$ for all $n\ge 4$. The case $n=2, 3$ can be checked directly. 
		
	\end{proof}

	\bibliographystyle{amsplain}

	\newpage
	
	\noindent John Man Shun Ma \\
	Department of Mathematical Science, University of Copenhagen \\
	Universitetsparken 5, DK-2100 Copenhagen \O, Denmark\\
	email: jm@math.ku.dk
	
	\bigskip
	
	\noindent Ali Muhammad \\
	Department of Mathematical Science, University of Copenhagen \\
	Universitetsparken 5, DK-2100 Copenhagen \O, Denmark\\
	email: ajhm@math.ku.dk
	
	\bigskip
	
	\noindent Niels Martin M\o ller \\
	Department of Mathematical Science, University of Copenhagen \\
	Universitetsparken 5, DK-2100 Copenhagen \O, Denmark\\
	email: nmoller@math.ku.dk
	
\end{document}